\documentclass[12pt]{article}
\usepackage{amsmath, amsthm, amsfonts, amssymb}
\usepackage{mathrsfs}
\usepackage{bbm}
\usepackage{amssymb}
\usepackage{txfonts}

\newlength{\defbaselineskip}
\newcounter{marnote}

\newcommand{\setlinespacing}[1]%
           {\setlength{\baselineskip}{#1 \defbaselineskip}}

\theoremstyle{plain}
\newtheorem{theorem}{Theorem}[section]

\newtheorem{lemma}[theorem]{Lemma}
\newtheorem{prop}[theorem]{Proposition}
\theoremstyle{definition}

\theoremstyle{remark}
\newtheorem{remark}{Remark}[section]
\numberwithin{equation}{section}

\begin{document}

% ------------------------------------------------------------------------

\title{Asymptotics of the gradient of solutions to the perfect conductivity problem}

\author{HaiGang
Li\footnote{School of Mathematical Sciences, Beijing Normal University, Laboratory of Mathematics and Complex Systems, Ministry of
   Education, Beijing 100875,
China. Corresponding author. Email:
hgli@bnu.edu.cn.}\quad YanYan Li\footnote{Department of
Mathematics, Rutgers University, 110 Frelinghuysen Rd, Piscataway,
NJ 08854, USA. Email: yyli@math.rutgers.edu.}\quad and\quad ZhuoLun Yang\footnote{Department of
Mathematics, Rutgers University, 110 Frelinghuysen Rd, Piscataway,
NJ 08854, USA. Email: zy110@math.rutgers.edu.}\quad }

\date{}

\maketitle

\begin{abstract}
In the perfect conductivity problem of composite material, the gradient of solutions can be arbitrarily large when two inclusions are located very close. To characterize the singular behavior of the gradient in the narrow region between two inclusions, we capture the leading term of the gradient and give a fairly sharp description of such asymptotics. 

\end{abstract}

\section{Introduction and main results}

It is important from an engineering point of view to study gradient estimates for solutions to a class of elliptic equations of divergence form with piecewise constant coefficients, which models the
conductivity problem of a composite material, frequently consisting of inclusions and background media. When the conductivity of inclusions degenerates to be infinity, we call it a perfect conductivity problem. It is known that the electric field, expressed as the gradient, in the the narrow region between inclusions may become arbitrarily large when the distance between two inclusions tends to zero. In this paper we characterize such blow-up rates of the gradient with respect to the distance and establish its asymptotic formula in dimensions two and three, two physically relevant dimensions, for two adjacent general convex inclusions.

Before stating our results, we first describe the nature of our domains. Let $\Omega\subset\mathbb{R}^{n}$, $n=2,3$, be a bounded open set with $C^{2}$ boundary, and let $D_{1}^{*}$ and $D_{2}^{*}$
be two open sets whose closure belonging to $\Omega$, touching at the origin with the inner normal of $\partial{D}_{1}^{*}$ being the positive $x_{n}$-axis. We write variable $x$ as $(x',x_{n})$. Translating $D_{1}^{*}$ and $D_{2}^{*}$ by $\frac{\varepsilon}{2}$ along $x_{n}$-axis, we obtain
$$D_{1}^{\varepsilon}:=D_{1}^{*}+(0',\frac{\varepsilon}{2}),\quad\mbox{and}\quad\,D_{2}^{\varepsilon}:=D_{2}^{*}-(0',\frac{\varepsilon}{2}).$$
When there is no possibility of confusion, we drop the superscripts $\varepsilon$ and denote $D_{1}:=D_{1}^{\varepsilon}$ and $D_{2}:=D_{2}^{\varepsilon}$.

The conductivity problem can be modeled by the following boundary value
problem of the scalar equation with piecewise constant coefficients
\begin{equation}\label{equk}
\begin{cases}
\mathrm{div}\Big(a_{k}(x)\nabla{u}_{k}\Big)=0&\mbox{in}~\Omega,\\
u_{k}=\varphi(x)&\mbox{on}~\partial\Omega,
\end{cases}
\end{equation}
where $\varphi\in{C}^{2}(\partial\Omega)$ is given, and
$$a_{k}(x)=
\begin{cases}
k\in[0,1)\cup(1,\infty]&\mbox{in}~D_{1}\cup{D}_{2},\\
1&\mbox{in}~\widetilde{\Omega}:=\Omega\setminus\overline{D_{1}\cup{D}_{2}}.
\end{cases}
$$
When $k$ is away from 0 and $\infty$, the gradient of the solution of \eqref{equk}, $\nabla{u}_{k}$, is bounded by a constant, independent of the distance $\varepsilon$. Babu\v{s}ka, Andersson, Smith, and Levin \cite{basl} computationally analyzed the damage and fracture in fiber composite materials where the Lam\'{e} system is used. They observed numerically that $|\nabla{u}_{k}|$ remains bounded when the distance $\varepsilon$ tends to zero. Bonnetier and Vogelius \cite{bv} proved that $|\nabla{u}_{k}|$ remains bounded for touching disks $D_{1}$ and $D_{2}$ in dimension $n=2$. The bound depends on the value of $k$. Li and Vogelius  \cite{lv} extended the result to general divergence form second order elliptic equations with piecewise H\"older continuous coefficients in all dimensions, and they proved that $|\nabla{u}_{k}|$ remains bounded as $\varepsilon\rightarrow0$.
They also established  stronger, $\varepsilon$-independent,  $C^{1,\alpha}$ estimates for solutions in the closure of each of the regions $D_1$, $D_2$ and $\widetilde \Omega$. This extension covers domains $D_{1}$ and $D_{2}$ of arbitrary smooth shapes. Li and Nirenberg \cite{ln} extended the results in \cite{lv} to general divergence form  second order elliptic systems including systems of elasticity.

In this paper, we consider the perfect conductivity problem when $k =+\infty$. It was proved by Ammari, Kang and Lim \cite{akl} and Ammari, Kang, H. Lee, J. Lee and Lim \cite{aklll} that, when $D_{1}$ and $D_{2}$ are disks of comparable radii embedded in $\Omega=\mathbb{R}^{2}$, the blow-up rate of the gradient of the solution to the perfect conductivity problem is $\varepsilon^{-1/2}$ as $\varepsilon$ goes to zero; with the lower bound given in \cite{akl} and the upper bound given in \cite{aklll}. Yun in \cite{y1,y2} generalized the above mentioned result by establishing the same lower bound, $\varepsilon^{-1/2}$, for two strictly convex subdomains in $\mathbb{R}^{2}$. More finer results in this line, see \cite{akllz,ly}. Bao, Li and Yin \cite{bly1} introduced a linear functional $Q_{\varepsilon}[\varphi]$ and obtained the optimal bounds
$$
\frac{\rho_{n}(\varepsilon)|Q_{\varepsilon}[\varphi]|}{C\varepsilon}\leq\|\nabla{u}\|_{L^{\infty}(\widetilde{\Omega})}\leq\frac{C\rho_{n}(\varepsilon)|Q_{\varepsilon}[\varphi]|}{\varepsilon}+C\|\varphi\|_{C^{2}(\partial\Omega)},
$$
where $C$ is independent of $\varepsilon$ or $\varphi$, and
$$\rho_{n}(\varepsilon)=
\begin{cases}
\sqrt{\varepsilon},~~&\mbox{for}~n=2;\\
|\log\varepsilon|^{-1},&\mbox{for}~n=3;\\
1,&\mbox{for}~n\geq4.
\end{cases}
$$
It may happen that for some $\varphi$, $|Q_{\varepsilon}[\varphi]|$ has positive lower and upper bounds independent of $\varepsilon$. It may also happen that for some $\varphi\nequiv0$ (independent of $\varepsilon$), $Q_{\varepsilon}[\varphi]=0$. A similar result for $p$-Laplace equation was investigated by Gorb and Novikov \cite{gn}. In particular, for $p=2$, they proved that
\begin{align*}
\lim_{\varepsilon\rightarrow0}\frac{\varepsilon\left\|\nabla{u}\right\|_{L^{\infty}(\widetilde{\Omega})}}{\rho_{n}(\varepsilon)}=\frac{\mathcal{R}_{o}}{C_{o}},\qquad~~\mbox{for}~~n=2,3,
\end{align*}
where $\mathcal{R}_{o}$ is a constant multiple of $Q_{\varepsilon}[\varphi]$, $C_{o}$ is an explicitly computable constant. The rate at which the $L^{\infty}$ norm of the gradient of a special solution for two identical circular inclusions in $\mathbb{R}^{2}$ has been shown in \cite{keller} to be $\varepsilon^{-1/2}$, see also \cite{bc,m}.

After knowing the blow-up rate of $|\nabla{u}|$ with respect to $\varepsilon$, it is desirous and important from the viewpoint of practical applications in engineering to capture such blow-up. Recently, Kang, Lim and Yun \cite{kly} characterize asymptotically the singular part of the solution for two adjacent circular inclusions $B_{1}$ and $B_{2}$ in $\mathbb{R}^{2}$ of radius $r_{1}$ and $r_{2}$ with $\varepsilon$ apart,
$$u(x)=\frac{2r_{1}r_{2}}{r_{1}+r_{2}}(\vec{n}\cdot\nabla{H})(\,p)\Big(\ln|x-p_{1}|-\ln|x-p_{2}|\Big)+g(x),$$
for $x\in\mathbb{R}^{2}\setminus(B_{1}\cup{B}_{2})$, where $H$ is a given entire harmonic function in $\mathbb{R}^{2}$, $p_{1}\in{B}_{1}$ and $p_{2}\in{B}_{2}$ are the fixed point of $R_{1}R_{2}$ and $R_{2}R_{1}$ respectively, $R_{j}$ is the reflection with respect to $\partial{B}_{j}$, $j=1,2$, $\vec{n}$ is the unit vector in the direction of $p_{2}-p_{1}$, and $p$ is the middle point of the shortest line segment connecting $\partial{B}_{1}$ and $\partial{B}_{2}$, and $|\nabla g(x)|$ is bounded independent of $\varepsilon$ on any bounded subset of $\mathbb{R}^{2}\setminus(B_{1}\cup{B}_{2})$. Then 
$$\nabla u(x)=\frac{2r_{1}r_{2}}{r_{1}+r_{2}}(\vec{n}\cdot\nabla{H})(\,p)\Big(\frac{1}{|x-p_{1}|}-\frac{1}{|x-p_{2}|}\Big)+\nabla g(x).$$
In $\mathbb{R}^{3}$, an analogous estimate is obtained by Kang, Lim, and Yun in \cite{kly2} in the narrow region between two balls with the same radius $r$ and when $\sqrt{x_{1}^{2}+x_{2}^{2}}\leq r|\log \varepsilon|^{-2}$. Ammari, Ciraolo, Kang, Lee, Yun \cite{ackly} extended the result in \cite{kly} to the case that inclusions $D_{1}$ and $D_{2}$ are strictly convex domains in $\mathbb{R}^{2}$. For two adjacent spherical inclusions in $\mathbb{R}^{3}$, it was studied by Kang, Lim and Yun \cite{kly2}. Bonnetier and Triki \cite{bt} derived the asymptotics of the eigenvalues of the Poincar\'e variational problem as the distance between the inclusions tends to zero.  The gradient estimates for Lam\'e system with partially infinite coefficients were recently obtained in \cite{bjl,bll,bll2}. For more related works, see \cite{adkl,bly2,bgn,dl,gorb2,ll,lx,lyu} and the references therein.

In this paper, we obtain estimates for perfect conductivity problems in bounded domains in $\mathbb{R}^{n}$, $n=2,3$, analogous to \cite{kly,kly2} in the whole space.  Our estimates in bounded domains in $\mathbb{R}^{3}$ improve those  in \cite{kly2} with a higher order asymptotic expansion. One of the main ingredients in achieving these is
an asymptotic expansion of the Dirichlet  energy of the harmonic function $v_{i}$ in $\widetilde{\Omega}$ satisfying $v_{i}=1$ on $\partial{D}_{i}$, and $v_{i}=0$ on $\partial\widetilde{\Omega}\setminus\partial{D}_{i}$, defined by the following
\begin{equation}\label{equv1}
\begin{cases}
\Delta{v}_{i}=0&\mbox{in}~\widetilde{\Omega},\\
v_{i}=\delta_{ij}&\mbox{on}~\partial{D}_{j},~i,j=1,2,\\
v_{i}=0&\mbox{on}~\partial\Omega.
\end{cases}
\end{equation}
Our method in deriving the asymptotics of the gradients are very different from that in \cite{ackly,kly,kly2}.

We assume that near the origin, $\partial{D}_{i}^{*}$ are respectively the graph of two $C^{2}$ functions $h_{1}$ and $h_{2}$, and for some $R_{0}, \kappa>0$, 
$$h_{1}(x')>h_{2}(x'),\quad\mbox{for}~~0<|x'|<R_{0},$$
\begin{equation}\label{h1h20}
h_{1}(0')=h_{2}(0')=0,\quad\nabla_{x'}h_{1}(0')=\nabla_{x'}h_{2}(0')=0,
\end{equation}
\begin{equation}\label{h1h21}
\nabla^{2}_{x'}(h_{1}(0')-h_{2}(0'))\geq\kappa I,
\end{equation}
where $I$ denotes the $(n-1)\times(n-1)$ identity matrix.

Here is the above mentioned ingredient, which has its independent interest.

\begin{theorem}\label{thm_energy}
Assume the above with $n=2,3$, $\partial{D}_{i}^{*}$ and $\partial \Omega$ are of ${C}^{k,1}$, $k\geq 3$. Let $v_{i}\in{H}^{1}(\widetilde{\Omega})$ be the solution of \eqref{equv1}, $i=1,2$. There exist $\varepsilon$-independent constants $M_{i}$, $i=1,2$,  and $C$, such that
\begin{equation*}\label{a1i}
\left|~\int_{\widetilde{\Omega}}|\nabla{v}_{i}|^{2}
-\Big(\frac{\kappa_{n}}{\rho_{n}(\varepsilon)}+M_{i}~\Big)~\right|\leq\,
\left\{ \begin{aligned}
&C\varepsilon^{\frac{1}{4}-\frac{1}{2k}}, && \mbox{if}~n = 2,\\
&C\varepsilon^{\frac{1}{2}-\frac{1}{2k}}|\log \varepsilon|, &&\mbox{if}~ n = 3,
\end{aligned}\right.
\quad\,i=1,2,
\end{equation*}
where  $\kappa_{2}=
\frac{\sqrt{2}\pi}{\sqrt{\lambda_{1}}}$, 
$\kappa_{3}=\frac{2\pi}{\sqrt{\lambda_{1}\lambda_{2}}}$, and $\lambda_{1}$ and $\lambda_{2}$ are the eigenvalues of $\nabla^{2}_{x'}(h_{1}-h_{2})(0')$.
\end{theorem}

Consider the perfect conductivity problem in the bounded domain $\Omega$:
\begin{equation}\label{equinfty}
\begin{cases}
\Delta{u}=0&\mbox{in}~\widetilde{\Omega},\\
\nabla{u}=0&\mbox{on}~\overline{D}_{i},~i=1,2,\\
u|_{+}=u|_{-}&\mbox{on}~\partial{D}_{i},~i=1,2,\\
\int_{\partial{D}_{i}}\frac{\partial{u}}{\partial\nu^{-}}=0&i=1,2,\\
u=\varphi(x)&\mbox{on}~\partial\Omega,
\end{cases}
\end{equation}
where $\varphi\in{C}^{0}(\partial\Omega)$, and for $x\in\partial{D}_{i}$
$$\frac{\partial{u}}{\partial\nu^{-}}(x):=\lim_{t\rightarrow0^{+}}\frac{u(x)-u(x+t\nu)}{t}.$$
Here and throughout this paper $\nu$ is the outward unit normal to
the domain and the subscript $\pm$ indicates the limit from outside and inside the domain, respectively. $u$ is the weak limit of $u_{k}\in{H}^{1}(\Omega)$, the solution of \eqref{equk}, as $k\rightarrow+\infty$. The existence, uniqueness and regularity of solutions to \eqref{equinfty} can be found in the appendix of \cite{bly1}.

We rewrite \eqref{equinfty} as
\begin{equation}\label{equinfty2}
\begin{cases}
\Delta{u}=0&\mbox{in}~\widetilde{\Omega},\\
u=C_{i}&\mbox{on}~\partial{D}_{i},~i=1,2,\\
\int_{\partial{D}_{i}}\frac{\partial{u}}{\partial\nu^{-}}=0&i=1,2,\\
u=\varphi(x)&\mbox{on}~\partial\Omega,
\end{cases}
\end{equation}
where $C_{1}$ and $C_{2}$ are constants uniquely determined by the third line. As in \cite{bly1}, we decompose the solution $u$ of \eqref{equinfty2} as follows
\begin{equation}\label{decomposition_u}
u(x)=C_{1}v_{1}(x)+C_{2}v_{2}(x)+v_{0}(x),\quad\mbox{in}~~\widetilde{\Omega},
\end{equation}
where $v_{1},v_{2}$ are defined by \eqref{equv1} and $v_{0}$ is the solution of
\begin{equation*}%\label{equv0}
\begin{cases}
\Delta{v}_{0}=0&\mbox{in}~\widetilde{\Omega},\\
v_{0}=0&\mbox{on}~\partial{D}_{1}\cup\partial{D}_{2},\\
v_{0}=\varphi(x)&\mbox{on}~\partial\Omega.
\end{cases}
\end{equation*}

 For $0\leq\,r\leq\,R_{0}$, let
$$ \Omega_r:=\left\{(x',x_{n})\in \mathbb{R}^{n}~\big|~-\frac{\varepsilon}{2}+h_{2}(x')<x_{n}<\frac{\varepsilon}{2}+h_{1}(x'),~|x'|<r\right\}.$$
In order to obtain the asymptotic expansion of $v_{1}$, we introduce an auxiliary function $\bar{u}\in{C}^{k,1}(\widetilde{\Omega})$ , such that $\bar{u}=1$ on $\partial{D}_{1}$, $\bar{u}=0$ on $\partial{D}_{2}\cup\partial\Omega$,
\begin{align}\label{ubar}
\bar{u}(x)
=\frac{x_{n}-h_{2}(x')+\frac{\varepsilon}{2}}{\varepsilon+h_{1}(x')-h_{2}(x')},\quad\mbox{in}~~\Omega_{R_{0}},
\end{align}
and
\begin{equation}\label{nablau_bar_outside}
\|\bar{u}\|_{C^{k,1}(\mathbb{R}^{n}\setminus \Omega_{R_{0}})}\leq\,C.
\end{equation}
Similarly, we can study the asymptotic expansion of $v_2$ through an auxiliary function $\tilde{u}:= 1 - \bar{u}$ in $\Omega_{R_0}$, with $\|\tilde{u}\|_{C^{k,1}(\mathbb{R}^{n}\setminus \Omega_{R_{0}})}\leq\,C$.
Recalling the assumption \eqref{h1h20} and \eqref{h1h21}, a direct computation yields
\begin{equation}\label{nablau_bar_inside}
\left|\partial_{x'}\bar{u}(x)\right|\leq\frac{C|x'|}{\varepsilon+|x'|^{2}},\quad
\partial_{x_{n}}\bar{u}(x)=\frac{1}{\varepsilon+h_{1}(x')-h_{2}(x')},\quad~x\in\Omega_{R_0}.
\end{equation}

Now define a linear functional $Q$ and a constant $\Theta$ as follows:
\begin{equation}\label{def_Q*}
Q[\varphi]:=\int_{\partial{D}_{1}^{*}}\frac{\partial{v}_{0}^{*}}{\partial\nu^{-}}
\int_{\partial{\Omega}}\frac{\partial{v}_{2}^{*}}{\partial\nu}
-\int_{\partial{D}_{2}^{*}}\frac{\partial{v}_{0}^{*}}{\partial\nu^{-}}\int_{\partial{\Omega}}\frac{\partial{v}_{1}^{*}}{\partial\nu},
\end{equation}
and
\begin{equation}\label{def_beta*}
\Theta:=-\kappa_{n}\int_{\partial{\Omega}}\frac{\partial({v}_{1}^{*}+ {v}_{2}^{*})}{\partial\nu}=\kappa_{n}\int_{\widetilde\Omega^{*}}\big|\nabla({v}_{1}^{*}+ {v}_{2}^{*})\big|^{2},
\end{equation}
where  $\kappa_{2}=\frac{\sqrt{2}\pi}{\sqrt{\lambda_{1}}}
$, 
$\kappa_{3}=\frac{2\pi}{\sqrt{\lambda_{1}\lambda_{2}}}$, and for $x\in\partial\Omega$
$$\frac{\partial{u}}{\partial\nu}(x):=\lim_{t\rightarrow0^{+}}\frac{u(x)-u(x-t\nu)}{t}.$$
Note that $\Theta/\kappa_{n}$ is the condenser capacity of $\partial{D}_{1}^{*}\cup\partial{D}_{2}^{*}$ relative to $\partial\Omega$. 
Here  ${v}_{1}^{*},{v}_{2}^{*}$ and ${v}_{0}^{*}$ are defined, respectively, by
\begin{equation}\label{equv1*}
\begin{cases}
\Delta{v}_{i}^{*}=0&\mbox{in}~\widetilde{\Omega}^{*}:=\Omega\setminus\overline{D_{1}^{*}\cup{D}_{2}^{*}},\\
v_{i}^{*}=\delta_{ij}&\mbox{on}~\partial{D}_{j}^{*}\setminus\{0\},\\
v_{i}^{*}=0&\mbox{on}~\partial\Omega,
\end{cases}~~i=1,2,
\quad\mbox{and}\quad
\begin{cases}
\Delta{v}_{0}^{*}=0&\mbox{in}~\widetilde{\Omega}^{*},\\
v_{0}^{*}=0&\mbox{on}~\partial{D}_{1}^{*}\cup\partial{D}_{2}^{*},\\
v_{0}^{*}=\varphi(x)&\mbox{on}~\partial\Omega.
\end{cases}
\end{equation}
The well-definedness and the boundedness that $0\leq\,v_{1}^{*},v_{2}^{*}\leq1$ can be seen from lemma 3.1 of \cite{bly1} and above that.

We have the asymptotic expression of $\nabla{u}$ in the narrow region between $D_{1}$ and $D_{2}$ as follows:
\begin{theorem}\label{thm1}
Let $\Omega, D_{1}^{*}, D_{2}^{*}$ be defined as the above. For $\varphi\in{C}^{0}(\partial{\Omega})$,
let $u\in{H}^{1}(\Omega)\cap{C}^{1}(\overline{\Omega}\setminus{(D_{1}\cup{D}_{2})})$
be the solution to \eqref{equinfty}. Then  
\begin{itemize}
\item[(i)] if $n=2$, $\partial D_{i}^{*}$ and $\partial \Omega$ are of $C^{3,1}$,
\begin{equation}\label{equ_thm1}
\nabla{u}=\frac{Q[\varphi]\sqrt{\varepsilon}}{\Theta}
\nabla\bar{u}+O(1)\|\varphi\|_{C^{0}(\partial\Omega)},\quad\,\mbox{in}~~\widetilde{\Omega};
\end{equation}
\item[(ii)] if $n=3$, $\partial D_{i}^{*}$ and $\partial \Omega$ are of $C^{k,1}$, $k \ge 3$, there exists a positive $\varepsilon-$independent constant $\widetilde{M}$, such that
\begin{equation}\label{equ2_thm1}
\nabla{u}=\frac{Q[\varphi]}{\Theta}\left(\frac{1}{|\log\varepsilon|-\widetilde{M}}+O(1)\varepsilon^{\frac{1}{2}-\frac{1}{2k}}|\log \varepsilon|^{-1}\right)
\nabla\bar{u}+O(1)\|\varphi\|_{C^{0}(\partial\Omega)},\quad\,\mbox{in}~~\widetilde{\Omega},
\end{equation}
\end{itemize}
where $O(1)$ denotes some quantity satisfying $|O(1)|\leq\,C$ for some $\varepsilon-$independent constant $C$.
\end{theorem}

The rest of this paper is organized as follows. In section \ref{sec_thm1}, we first reduce the proof of Theorem 1.2 to Proposition \ref{prop_v1-u_bar} and Proposition \ref{lem_Qbeta} below, one for the estimates of $\|\nabla(v_{1}-\bar{u})\|_{L^{\infty}(\widetilde{\Omega})}$, the other for the estimate of $C_{1}-C_{2}$, then prove them in Subsection \ref{subsection2.2} and Subsection \ref{subsec_Qbeta}, respectively. Finally, we give the proof of Theorem \ref{thm_energy} in Section \ref{sec_thm_energy}. 

\vspace{.5cm} 

\noindent{\bf{\large Acknowledgements.}} The research of H.G. Li is partially supported by  NSFC (1157 1042), (11631002) and Fok Ying Tung Education Foundation (151003). H.G. Li would like to thank Professor Jiguang Bao for his suggestions and constant encouragement to this work. The research of Y.Y. Li is partially supported by NSF grant DMS-1501004. 
We thank the referees for helpful suggestions which improve the exposition.

\section{The proof of Theorem \ref{thm1}}\label{sec_thm1}

\subsection{The strategy to prove Theorem \ref{thm1}}

This section is devoted to the proof of Theorem \ref{thm1}. We make use of the energy method to single out the singular term of $\nabla{u}$. 
We only need to prove \eqref{equ_thm1} and \eqref{equ2_thm1} with $ \| \varphi \|_{C^0(\partial \Omega)}$ replaced by
$\| \varphi \|_{C^{2}(\partial\Omega)}$.   Indeed, since 
 $u_k$ in \eqref{equk} satisfies  $\|u_k\|_{L^\infty(\Omega)} \le \| \varphi\|_{C^{0}(\partial\Omega)}$, we have by the convergence of $u_k$ to $u$ (see Appendix in \cite{bly1}), 
$\| u \|_{L^\infty (\Omega)} \le \| \varphi \|_{C^0(\partial \Omega)}.$
Taking a slightly smaller domain $\Omega_1 \subset \subset \Omega$, 
 then $\varphi_1 := u~\big|_{\partial \Omega_1}$ satisfies
$\| \varphi_1 \|_{C^{2}(\partial\Omega_1)} \le C \| u \|_{L^\infty(\Omega)} \le C \| \varphi \|_{C^{0}(\partial\Omega)}$ in view of interior derivative estimates
for harmonic functions.   The desired identity \eqref{equ_thm1} follows by working with $u$, $\Omega_1$ and $\varphi_1$. Without loss of generality, we assume that $\|\varphi\|_{C^{2}(\partial\Omega)}=1$, by considering $u/\|\varphi\|_{C^{2}(\partial\Omega)}$ if $\|\varphi\|_{C^{2}(\partial\Omega)}>0$. If $\varphi~\big|_{\partial\Omega}=0$ then $u\equiv0$.

From \eqref{decomposition_u}, we have
\begin{equation*}%\label{decomposition_u2}
\nabla{u}=(C_{1}-C_{2})\nabla{v}_{1}+C_{2}\nabla({v}_{1}+{v}_{2})+\nabla{v}_{0}.
\end{equation*}
Noting that $u=C_{i}$ on $\partial{D}_{i}$ and $\|u\|_{H^{1}(\widetilde\Omega)}\leq\,C$ (independent of $\varepsilon$), using the trace embedding theorem,  we have
\begin{equation}\label{C1C2}
|C_{1}|+|C_{2}|\leq\,C.
\end{equation}
Since $\Delta{v}_{0}=0$ in $\widetilde{\Omega}$ with $v_{0}=0$ on $\partial{D}_{1}\cup\partial{D}_{2}$, and $\Delta(v_{1}+v_{2}-1)=0$ in $\widetilde{\Omega}$ with $v_{1}+v_{2}-1=0$ on $\partial{D}_{1}\cup\partial{D}_{2}$, it follows from lemma 2.3 in \cite{bly1} (or theorem 1.1 in \cite{llby}) and the standard elliptic theory that
\begin{equation}\label{v0_bdd}
\big\|\nabla{v}_{0}\big\|_{L^{\infty}(\widetilde{\Omega})}
\leq\,C,\quad\mbox{and}\quad\big\|\nabla(v_{1}+v_{2})\big\|_{L^{\infty}(\widetilde{\Omega})}
\leq\,C.
\end{equation}

Recalling the definition of $\bar{u}$ in $\Omega_{R_{0}}$, \eqref{ubar}, we first prove that the $L^{\infty}$ norm of $\nabla(v_{1}-\bar{u})$ is bounded.

\begin{prop}\label{prop_v1-u_bar}
Under the assumptions of Theorem \ref{thm_energy}, let $v_{1}\in{H}^1(\widetilde{\Omega})$ be the
weak solution of \eqref{equv1}. Then
\begin{equation}\label{nabla_w_i0}
\|\nabla(v_{1}-\bar{u})\|_{L^{\infty}(\widetilde{\Omega})}\leq\,C.
\end{equation}
Consequently,
\begin{equation}\label{nabla_v_i0}
\frac{1}{C(\varepsilon+|x'|^{2})}\leq|\nabla{v}_{1}(x)|\leq\frac{C}{\varepsilon+|x'|^{2}},\quad~x\in\Omega_{R_{0}},
\end{equation}
and
\begin{equation}\label{nabla_v_i_out}
\|\nabla{v}_{1}\|_{L^{\infty}(\widetilde{\Omega}\setminus\Omega_{R_{0}})}\leq\,C.
\end{equation}
\end{prop}

The proof will be given in Section \ref{subsection2.2}. 

On the other hand, from the third line of \eqref{equinfty2} and \eqref{decomposition_u}, the constants $C_{1}$ and $C_{2}$ are determined by the following linear system
\begin{equation}\label{sysC1C2*}
\begin{cases}
\displaystyle C_{1}\int_{\partial{D}_{1}}\frac{\partial{v}_{1}}{\partial\nu^{-}}+C_{2}\int_{\partial{D}_{1}}\frac{\partial{v}_{2}}{\partial\nu^{-}}+
\int_{\partial{D}_{1}}\frac{\partial{v}_{0}}{\partial\nu^{-}}=0,\\\\
\displaystyle C_{1}\int_{\partial{D}_{2}}\frac{\partial{v}_{1}}{\partial\nu^{-}}+C_{2}\int_{\partial{D}_{2}}\frac{\partial{v}_{2}}{\partial\nu^{-}}+
\int_{\partial{D}_{2}}\frac{\partial{v}_{0}}{\partial\nu^{-}}=0.
\end{cases}
\end{equation}
Similarly, as in \cite{bly1}, we denote
$$a_{ij}:=\int_{\partial{D}_{1}}\frac{\partial{v}_{1}}{\partial\nu^{-}},\quad\quad\,b_{i}:=-\int_{\partial{D}_{i}}\frac{\partial{v}_{0}}{\partial\nu^{-}},\quad\,i,j=1,2.$$
Then \eqref{sysC1C2*} can be written as
\begin{equation*}
\begin{cases}
\displaystyle a_{11}C_{1}+a_{12}C_{2}=b_{1},\\
\displaystyle a_{21}C_{1}+a_{22}C_{2}=b_{2}.
\end{cases}
\end{equation*}
By using Cramer's rule, we have
$$C_{1}=\frac{\begin{vmatrix}
~b_{1}&a_{12}~\\
b_{2}&a_{22}
\end{vmatrix}}{\begin{vmatrix}
~a_{11}&a_{12}~\\
a_{21}&a_{22}
\end{vmatrix}},\quad\,C_{2}=\frac{\begin{vmatrix}
~a_{11}&b_{1}~\\
a_{21}&b_{2}
\end{vmatrix}}{\begin{vmatrix}
~a_{11}&a_{12}~\\
a_{21}&a_{22}
\end{vmatrix}},\quad\mbox{and}\quad\,C_{1}-C_{2}=\frac{\begin{vmatrix}
~b_{1}&a_{11}+a_{12}~\\
b_{2}&a_{21}+a_{22}
\end{vmatrix}}{\begin{vmatrix}
~a_{11}&a_{12}~\\
a_{21}&a_{22}
\end{vmatrix}}.$$
By the Green's formula, it is easy to see that $a_{12}=a_{21}$, and
$$a_{11}+a_{12}=a_{11}+a_{21}=-\int_{\partial{\Omega}}\frac{\partial{v}_{1}}{\partial\nu},\quad\,a_{21}+a_{22}=a_{12}+a_{22}=-\int_{\partial{\Omega}}\frac{\partial{v}_{2}}{\partial\nu}.$$
In view of
$$\begin{vmatrix}
~a_{11}&a_{12}~\\
a_{21}&a_{22}
\end{vmatrix}=\begin{vmatrix}
~a_{11}&a_{11}+a_{12}~\\
a_{21}&a_{21}+a_{22}
\end{vmatrix},$$
and denoting
\begin{equation}\label{def_Q}
Q_{\varepsilon}[\varphi]:=\int_{\partial{D}_{1}}\frac{\partial{v}_{0}}{\partial\nu^{-}}
\int_{\partial{\Omega}}\frac{\partial{v}_{2}}{\partial\nu}
-\int_{\partial{D}_{2}}\frac{\partial{v}_{0}}{\partial\nu^{-}}\int_{\partial{\Omega}}\frac{\partial{v}_{1}}{\partial\nu},
\end{equation}
and
\begin{equation}\label{def_beta}
\Theta_{\varepsilon}:=-\left(\rho_{n}(\varepsilon)\int_{\partial{D}_{1}}\frac{\partial{v}_{1}}{\partial\nu^{-}}\right)\int_{\partial{\Omega}}\frac{\partial{v}_{2}}{\partial\nu}
+\left(\rho_{n}(\varepsilon)\int_{\partial{D}_{1}}\frac{\partial{v}_{2}}{\partial\nu^{-}}\right)\int_{\partial{\Omega}}\frac{\partial{v}_{1}}{\partial\nu},
\end{equation}
we have
$$C_{1}-C_{2}=\rho_{n}(\varepsilon)\frac{Q_{\varepsilon}[\varphi]}{\Theta_{\varepsilon}}.$$
The following asymptotic expansion of $\frac{Q_{\varepsilon}[\varphi]}{\Theta_{\varepsilon}}$ in term of $\rho_{n}(\varepsilon)$ is an essential part in the proof of Theorem \ref{thm1}.
\begin{prop}\label{lem_Qbeta}
Under the same assumptions of Theorem \ref{thm1}, let $Q_{\varepsilon}[\varphi]$ and $\Theta_{\varepsilon}$ be defined by \eqref{def_Q} and \eqref{def_beta}, $Q[\varphi]$ and $\Theta$ be defined by \eqref{def_Q*} and \eqref{def_beta*}, respectively. Then
\begin{align*}
\frac{Q_{\varepsilon}[\varphi]}{\Theta_{\varepsilon}}-\frac{Q[\varphi]}{\Theta}
=\frac{Q[\varphi]}{\Theta}\frac{\widetilde{M}\rho_{n}(\varepsilon)}{1-\widetilde{M}\rho_{n}(\varepsilon)}
+E_n(\varepsilon)\rho_{n}(\varepsilon),\quad\,n=2,3,
\end{align*}
where $\widetilde{M}$ is an $\varepsilon-$independent constant, and
\begin{align*}
E_n(\varepsilon) = 
\begin{cases}
O\left(\varepsilon^{\frac{1}{12}} \right), &\mbox{if}~n=2, \\
O\left(\varepsilon^{\frac{1}{2} - \frac{1}{2k}} |\log \varepsilon| \right), &\mbox{if}~n=3. 
\end{cases}
\end{align*}
\end{prop}

The proof of Proposition \ref{lem_Qbeta} will be given in Section \ref{subsec_Qbeta}. We are now in a position to prove Theorem \ref{thm1} by using Proposition \ref{prop_v1-u_bar} and Proposition \ref{lem_Qbeta}.

\begin{proof}[Proof of Theorem \ref{thm1}] By using \eqref{C1C2}, \eqref{v0_bdd} and \eqref{nabla_w_i0},\begin{equation*}%\label{decomposition_u3}
\nabla{u}=(C_{1}-C_{2})\nabla\bar{u}+O(1).
\end{equation*}
It follows from Proposition \ref{lem_Qbeta} that
\begin{align*}
\frac{C_{1}-C_{2}}{\rho_{n}(\varepsilon)}-\frac{Q[\varphi]}{\Theta}& =\frac{Q_{\varepsilon}[\varphi]}{\Theta_{\varepsilon}}-\frac{Q[\varphi]}{\Theta}\\
&=\frac{Q[\varphi]}{\Theta}\frac{\widetilde{M}\rho_{n}(\varepsilon)}{1-\widetilde{M}\rho_{n}(\varepsilon)}+E_n(\varepsilon)\rho_{n}(\varepsilon).
\end{align*}
So that
\begin{align*}
C_{1}-C_{2}=&\rho_{n}(\varepsilon)\left(\frac{Q[\varphi]}{\Theta}+\frac{Q[\varphi]}{\Theta}\frac{\widetilde{M}\rho_{n}(\varepsilon)}{1-\widetilde{M}\rho_{n}(\varepsilon)}+E_n(\varepsilon)\rho_{n}(\varepsilon)\right)\\
=&\frac{Q[\varphi]}{\Theta}\rho_{n}(\varepsilon)\left(1+\frac{\widetilde{M}\rho_{n}(\varepsilon)}{1-\widetilde{M}\rho_{n}(\varepsilon)} + E_n(\varepsilon)\rho_{n}(\varepsilon)\right)\\
=&\frac{Q[\varphi]}{\Theta}\rho_{n}(\varepsilon)\left(\frac{1}{1-\widetilde{M}\rho_{n}(\varepsilon)}+E_n(\varepsilon)\rho_{n}(\varepsilon)\right).
\end{align*}
Thus, 
\begin{align*}
\nabla{u}(x)=&(C_{1}-C_{2})\nabla\bar{u}(x)+O(1)\\
=&\frac{Q[\varphi]}{\Theta}\rho_{n}(\varepsilon)\left(\frac{1}{1-\widetilde{M}\rho_{n}(\varepsilon)}
+E_n(\varepsilon)\rho_{n}(\varepsilon)\right)
\nabla\bar{u}(x)+O(1).
\end{align*}
Theorem 1.2 follows easily from the above and Proposition 2.1.
\end{proof}

\subsection{Proof of Proposition \ref{prop_v1-u_bar}}\label{subsection2.2}

\begin{proof}[Proof of Proposition \ref{prop_v1-u_bar}]

We denote
\begin{equation*}%\label{def_w}
w:=v_{1}-\bar{u}.
\end{equation*}
By the definition of $v_{1}$ in \eqref{equv1}, and the fact that $v_{1}=\bar{u}$ on $\partial{D}_{1}\cup\partial{D}_{2}\cup\partial\Omega$, we have
\begin{equation}\label{w20}
\begin{cases}
-\Delta{w}=\Delta\bar{u}&\mbox{in}~\widetilde{\Omega},\\
w=0&\mbox{on}~\partial\widetilde{\Omega}.
\end{cases}
\end{equation}
Recalling the definition of $\bar{u}$, \eqref{ubar} and \eqref{nablau_bar_outside}, 
\begin{equation}\label{nabla_u_out}
\|\bar{u}\|_{C^{k,1}(\widetilde\Omega\setminus\Omega_{R_{0}/2})}\leq\,C.\end{equation}
By standard elliptic theories, we know that
\begin{equation*}%\label{nabla_w_out}
|w|+\left|\nabla{w}\right|+\left|\nabla^{2}w\right|\leq\,C,
\quad\mbox{in}~~\widetilde\Omega\setminus\Omega_{R_{0}}.
\end{equation*}
Therefore, to show \eqref{nabla_w_i0}, we only need to prove
$$\left\|\nabla{w}\right\|_{L^{\infty}(\Omega_{R_{0}})}\leq\,C.$$

For $(z',z_{n})\in\Omega_{2R_{0}}$, denote
\begin{equation}\label{delta_x}
\delta(z'):=\varepsilon+h_{1}(z')-h_{2}(z').
\end{equation}
The rest of the proof is divided into three steps.

\noindent{\bf STEP 1.} Boundedness of the energy of $w$ in $\widetilde{\Omega}$:
\begin{equation}\label{energy_w}
\int_{\widetilde{\Omega}}\left|\nabla{w}\right|^{2}\leq\,C.
\end{equation}

By the maximum principle, $0<v_{1}<1$. Recalling the definition of $\bar{u}$, $\bar{u}$ is also bounded. Hence 
\begin{equation}\label{w_bdd}
\|w\|_{L^{\infty}(\widetilde{\Omega})}\leq\,C.
\end{equation}
A direct computation yields, 
$$|\partial_{x_{i}x_{j}}\bar{u}(x)|\leq\frac{C}{\varepsilon+|x'|^{2}},\quad|\partial_{x_{i}x_{n}}\bar{u}(x)|\leq\frac{C|x'|}{(\varepsilon+|x'|^{2})^{2}},
\quad\partial_{x_{n}x_{n}}\bar{u}(x)=0,\quad\mbox{for}~(x',x_{n})\in\Omega_{R_{0}}.$$
So that
\begin{equation}\label{delta_ubar}
|\Delta\bar{u}|\leq\frac{C}{\varepsilon+|x'|^{2}},\qquad\,x\in\Omega_{R_{0}}.
\end{equation}
Now multiply the equation in \eqref{w20} by $w$, integrate by parts, and make use of \eqref{nabla_u_out}, \eqref{w_bdd} and \eqref{delta_ubar}, 
\begin{align*}%\label{energy_w_1}
\int_{\widetilde{\Omega}}|\nabla{w}|^{2}
=\int_{\widetilde{\Omega}}w\left(\Delta\bar{u}\right)
\leq\|w\|_{L^{\infty}(\widetilde{\Omega})}\left(\int_{\Omega_{R_{0}}}|\Delta\bar{u}|+C\right)\leq\,C.
\end{align*}
Thus, \eqref{energy_w} is proved.

\noindent{\bf STEP 2.} Proof of
\begin{equation}\label{energy_w_inomega_z1}
\frac{1}{|\Omega_{\delta}(z')|}\int_{\Omega_{\delta}(z')}\left|\nabla{w}\right|^{2}dx\leq\,C,\quad\mbox{for}~~n\geq2,
\end{equation}
where
$$\Omega_{\delta}(z')=\left\{x\in\mathbb{R}^{n}\bigg|-\frac{\varepsilon}{2}+h_{2}(x')<x_{n}
<\frac{\varepsilon}{2}+h_{1}(x'),|x' - z'|< \delta \right\},$$
and $\delta:=\delta(z')$ is defined in \eqref{delta_x}.

The proof is in spirit similar to that in \cite{llby} and \cite{bll,bll2}, see in particular, the proof of proposition 3.2 in \cite{bll}. For reader's convenience, we outline the proof here. For $0<t<s<R_{0}$, let $\eta$ be a smooth cutoff function satisfying $\eta(x')=1$ if $|x'-z'|<t$, $\eta(x')=0$ if $|x'-z'|>s$, $0\leq\eta(x')\leq1$ if $t\leq|x'-z'|\leq\,s$, and $|\nabla_{x'}\eta(x')|\leq\frac{2}{s-t}$. Multiplying the equation in \eqref{w20} by $w\eta^{2}$ and integrating by parts
leads  to
\begin{align}\label{FsFt11}
\int_{\Omega_{t}(z')}|\nabla{w}|^{2}\leq\,\frac{C}{(s-t)^{2}}\int_{\Omega_{s}(z')}|w|^{2}
+(s-t)^{2}\int_{\Omega_{s}(z')}\left|\Delta\bar{u}\right|^{2}.
\end{align}

{\bf Case 1.} For $\sqrt{\varepsilon}\leq|z'|\leq\,R_{0}$. For $0<s<\frac{2|z'|}{3}$, note that
\begin{align}\label{energy_w_square}
\int_{\Omega_{s}(z')}|w|^{2}
\leq&\,C|z'|^{4}\int_{\Omega_{s}(z')}|\nabla{w}|^{2}, \quad\mbox{if}~\,0<s<\frac{2|z'|}{3}.
\end{align}
Substituting it into \eqref{FsFt11} and denoting
$$\widehat{F}(t):=\int_{\Omega_{t}(z')}|\nabla{w}|^{2},$$
we have
\begin{equation}\label{tildeF111}
\widehat{F}(t)\leq\,\left(\frac{C_{0}|z'|^{2}}{s-t}\right)^{2}\widehat{F}(s)+C(s-t)^{2}\int_{\Omega_{s}(z')}\left|\Delta\bar{u}\right|^{2},
\quad\forall~0<t<s<\frac{2|z'|}{3},
\end{equation}
where $C_0$ is a positive universal constant.

Let $k=\left[\frac{1}{4C_{0}|z'|}\right]$ and $t_{i}=\delta+2C_{0}i\,|z'|^{2}$, $i=0,1,2,\cdots,k$. Taking $s=t_{i+1}$ and $t=t_{i}$ in \eqref{tildeF111}, and in view of \eqref{delta_ubar}, 
\begin{equation}\label{integal_Lubar11}
\int_{\Omega_{t_{i+1}}(z')}\left|\Delta\bar{u}\right|^{2}\leq\int_{|x'-z'|<t_{i+1}}\frac{C}{\varepsilon+|x'|^{2}}dx'\leq\frac{Ct_{i+1}^{n-1}}{|z'|^{2}}\leq\,C(i+1)^{n-1}|z'|^{2(n-2)}.
\end{equation}
we obtain the iteration formula
$$\widehat{F}(t_{i})\leq\,\frac{1}{4}\widehat{F}(t_{i+1})+C(i+1)^{n-1}|z'|^{2n}.$$
After $k$ iterations, using (\ref{energy_w}),
\begin{eqnarray*}
\widehat{F}(t_{0}) \leq (1/4)^{k}\widehat{F}(t_{k})+C|z'|^{2n}\sum_{l=1}^{k}(1/4)^{l-1}l^{n-1}
\leq C|z'|^{2n}.
\end{eqnarray*}
This implies that
\begin{equation*}%\label{energy_w_in_omega_z}
\int_{\Omega_{\delta}(z')}|\nabla{w}|^{2}\leq\,C|z'|^{2n}.
\end{equation*}

{\bf Case 2.} For $0\leq|z'|\leq \sqrt{\varepsilon}$. Estimate \eqref{energy_w_square} becomes
 \begin{align*}%\label{energy_w_square_in}
\int_{\Omega_{s}(z')}|w|^{2}
\leq\,C\varepsilon^{2}\int_{\Omega_{s}(z')}|\nabla{w}|^{2},
\quad\mbox{if}~\,0<s<\sqrt{\varepsilon}.
\end{align*}
Estimate \eqref{tildeF111} becomes, in view of \eqref{FsFt11},
\begin{equation}\label{tildeF111_in}
\widehat{F}(t)\leq\,\left(\frac{C_{1}\varepsilon}{s-t}\right)^{2}\widehat{F}(s)+C(s-t)^{2}\int_{\Omega_{s}(z')}\left|\Delta\bar{u}\right|^{2},
\quad\forall~0<t<s<\sqrt{\varepsilon},
\end{equation}
where $C_1$ is another positive universal constant. Let $k=\left[\frac{1}{4C_{1}\sqrt{\varepsilon}}\right]$ and $t_{i}=\delta+2C_{1}i\varepsilon$, $i=0,1,2,\cdots,k$. Then
by \eqref{tildeF111_in} with $s=t_{i+1}$ and $t=t_{i}$, 
and using, instead of estimate \eqref{integal_Lubar11}, 
\begin{equation}\label{integal_Lubar11_in}
\int_{\Omega_{t_{i+1}}(z')}\left|\Delta\bar{u}\right|^{2}
\leq\int_{|x'-z'|<t_{i+1}}\frac{C}{\varepsilon+|x_{1}|^{2}}dx_{1}\leq\frac{Ct_{i+1}^{n-1}}{\varepsilon}\leq\,C(i+1)^{n-1}\varepsilon^{n-2},\quad\mbox{if}~\,0<s<\sqrt{\varepsilon}.
\end{equation}
we have
$$\widehat{F}(t_{i})\leq\,\frac{1}{4}\widehat{F}(t_{i+1})
+C(i+1)^{n-1}\varepsilon^{n}.$$
After $k$ iterations, using (\ref{energy_w}),
\begin{eqnarray*}
\widehat{F}(t_{0})
\leq (1/4)^{k}\widehat{F}(t_{k})
+C\sum_{l=1}^{k}(1/4)^{l-1}l^{n-1}\varepsilon^{n}
\leq\,
C(1/4)^{\frac{1}{C\sqrt{\varepsilon}}}
+C\varepsilon^{n}\leq\,C\varepsilon^{n}.
\end{eqnarray*}
This implies that
\begin{equation*}%\label{energy_w_in_omega_epsilon_n=2}
\int_{\Omega_{\delta}(z')}|\nabla{w}|^{2}\leq\,C\varepsilon^{n}.
\end{equation*}
In view of the definition of $\delta(z')$, \eqref{energy_w_inomega_z1} is proved.

\noindent{\bf STEP 3.} Proof of \eqref{nabla_w_i0}.

By using the following scaling and translating of variables
\begin{equation*}%\label{changeofvariant}
 \left\{
  \begin{aligned}
  &x'-z'=\delta y',\\
  &x_n=\delta y_n,
  \end{aligned}
 \right.
\end{equation*}
then $\Omega_{\delta}(z')$ becomes $Q_{1}$, where
$$Q_{r}=\left\{y\in\mathbb{R}^{n}\bigg|-\frac{\varepsilon}{2\delta}+\frac{1}{\delta}h_{2}(\delta{y}'+z')<y_{n}
<\frac{\varepsilon}{2\delta}+\frac{1}{\delta}h_{1}(\delta{y}'+z'),|y'|<r\right\},\quad\mbox{for}~~r\leq1,$$ and the top and
bottom boundaries respectively
become
$$
y_n=\hat{h}_{1}(y'):=\frac{1}{\delta}
\left(\frac{\varepsilon}{2}+h_{1}(\delta\,y'+z')\right),\quad|y'|<1,$$
and
$$y_n=\hat{h}_{2}(y'):=\frac{1}{\delta}\left(-\frac{\varepsilon}{2}
+h_{2}(\delta\,y'+z')\right), \quad |y'|<1.
$$
 Then
$$\hat{h}_{1}(0')-\hat{h}_{2}(0'):=\frac{1}{\delta}\left(\varepsilon+h_{1}(z')-h_{2}(z')\right)=1,$$
and by \eqref{h1h20},
$$|\nabla_{x'}\hat{h}_{1}(0')|+|\nabla_{x'}\hat{h}_{2}(0')|\leq\,C|z'|,
\quad|\nabla_{x'}^{2}\hat{h}_{1}(0')|+|\nabla_{x'}^{2}\hat{h}_{2}(0')|\leq\,C.$$
Since $R_{0}$ is small, $\|\hat{h}_{1}\|_{C^{1,1}((-1,1)^{n-1})}$ and $\|\hat{h}_{2}\|_{C^{1,1}((-1,1)^{n-1})}$ are small and $Q_{1}$ is essentially a unit square (or a unit cylinder for $n=3$) as far as applications of the Sobolev embedding theorem and classical $L^{p}$ estimates for elliptic equations are concerned.
Let
\begin{equation*}%\label{def_U}
U(y', y_n):=\bar{u}(z'+\delta{y}',\delta{y}_{n}),\quad\,W(y', y_n):=w(z'+\delta{y}',\delta{y}_{n}),
\quad\,y\in{Q}'_{1},
\end{equation*}
then by \eqref{w20},
\begin{align*}%\label{syswidew2}
-\Delta{W}
=\Delta{U},
\quad\quad\,y\in{Q_{1}},
\end{align*}
where
$$\left|\Delta{U}\right|=\delta^{2}\left|\Delta\bar{u}\right|.$$

Since $W=0$ on the top and
bottom boundaries of $Q_{1}$, using
the Poincar\'{e} inequality,
$$\left\|W\right\|_{H^{1}(Q_{1})}\leq\,C\left\|\nabla{W}\right\|_{L^{2}(Q_{1})}.$$
By $W^{2,p}$ estimates for elliptic equations (see e.g. \cite{gt}) and
the Sobolev embedding theorems, with $p>n$,
\begin{align*}
\left\|\nabla{W}\right\|_{L^{\infty}(Q_{1/2})}\leq\,
C\left\|W\right\|_{W^{2,p}(Q_{1/2})}
\leq\,C\left(\left\|\nabla{W}\right\|_{L^{2}(Q_{1})}+\left\|\Delta{U}\right\|_{L^{\infty}(Q_{1})}\right).
\end{align*}
It follows from $\nabla{W}=\delta\nabla{w}$ that
\begin{equation*}
\left\|\nabla{w}\right\|_{L^{\infty}(\Omega_{\delta/2}(z'))}\leq\,
C\left(\delta^{-n/2}\left\|\nabla{w}\right\|_{L^{2}(\Omega_{\delta}(z'))}
+\delta\left\|\Delta\bar{u}\right\|_{L^{\infty}(\Omega_{\delta}(z'))}\right).
%\label{AAA}
\end{equation*}
Using \eqref{delta_ubar}, \eqref{energy_w_inomega_z1}, and the definition of $\Omega_{\delta}(z')$, Proposition \ref{prop_v1-u_bar} is established.
\end{proof}

\begin{remark}\label{rem31}
We point out that the estimate involving $\Delta\bar{u}$ is very crucial in the above proof, such as \eqref{integal_Lubar11}, \eqref{integal_Lubar11_in} for $\int_{\Omega_{t_{i+1}}(z')}\left|\Delta\bar{u}\right|^{2}
$ and $\delta\left\|\Delta\bar{u}\right\|_{L^{\infty}(\Omega_{\delta}(z'))}$, so that it is essentially important to select an auxiliary function $\bar{u}$ to obtain appropriate estimates \eqref{delta_ubar}. 
\end{remark}

\subsection{The proof of Proposition \ref{lem_Qbeta}}\label{subsec_Qbeta}

Since
\begin{align*}
\frac{Q_{\varepsilon}[\varphi]}{\Theta_{\varepsilon}}-\frac{Q[\varphi]}{\Theta}
=&\frac{Q_{\varepsilon}[\varphi]-Q[\varphi]}{\Theta_{\varepsilon}}+\frac{Q[\varphi]}{\Theta}\frac{\Theta-\Theta_{\varepsilon}}{\Theta_{\varepsilon}},
\end{align*}
it follows that the proof of Proposition \ref{lem_Qbeta} can be reduced to the establishment of three Lemmas in the following.

\begin{lemma}\label{lem_24}
Let $\Theta$ and $\Theta_{\varepsilon}$ be defined as \eqref{def_beta*} and \eqref{def_beta}, respectively. There exists some universal constant $\delta_{0}>0$ such that
\begin{equation*}%\label{beta_lowbound}
\Theta\geq \delta_{0},
\end{equation*}
and $\lim\limits_{\varepsilon\to0}\Theta_{\varepsilon}=\Theta$. Consequently, for sufficiently small $\varepsilon$,
\begin{equation*}%\label{beta*_lowbound}
\Theta_{\varepsilon}\geq \delta_{0}/2.
\end{equation*}
\end{lemma}

\begin{lemma}\label{lem_25}
Let $\Theta$ and $\Theta_{\varepsilon}$ be defined as \eqref{def_beta*} and \eqref{def_beta}, respectively. Then
\begin{equation}\label{beta_difference}
\Theta-\Theta_{\varepsilon}=\rho_{n}(\varepsilon)\left(M_{1}\int_{\partial\Omega}\frac{\partial(v_{1}^{*}+v_{2}^{*})}{\partial\nu}+\Big(\int_{\partial{\Omega}}\frac{\partial{v}_{1}^{*}}{\partial\nu}\Big)^{2}\right)+E_n(\varepsilon)\rho_{n}(\varepsilon),
\end{equation}
where $M_{1}$ is the constant determined in Theorem \ref{thm_energy}. Consequently,
\begin{equation*}
\frac{\Theta-\Theta_{\varepsilon}}{\Theta}=\widetilde{M}\rho_{n}(\varepsilon)+E_n(\varepsilon)\rho_{n}(\varepsilon),
\end{equation*}
where
\begin{equation*}%\label{defM*}
\widetilde{M}:=-\frac{M_{1}}{\kappa_n}+\frac{(\alpha_{1}^{*})^{2}}{\Theta},\quad\, \alpha_{i}^{*}=\int_{\partial{\Omega}}\frac{\partial{v}_{i}^{*}}{\partial\nu},
\end{equation*}
which depend only on $D_{1}^{*},D_{2}^{*}$ and $\Omega$.
\end{lemma}

\begin{lemma}\label{lem_26}
Let $Q[\varphi]$ and $Q_{\varepsilon}[\varphi]$ be defined as \eqref{def_Q*} and \eqref{def_Q}, respectively. Then

\begin{equation}\label{Q_difference}
Q_{\varepsilon}[\varphi]-Q[\varphi]=\begin{cases}
O(\varepsilon^{3/4}), &\mbox{if}~n=2;\\
O(\varepsilon|\log\varepsilon|), &\mbox{if}~n=3.
\end{cases}
\end{equation}
\end{lemma}

We first prove Proposition \ref{lem_Qbeta} by using Lemma \ref{lem_24}--\ref{lem_26}, whose proofs will be given later.

\begin{proof}[Proof of Proposition \ref{lem_Qbeta}]
By Lemma \ref{lem_24}-\ref{lem_26}, for $n=2$,
\begin{align*}
\frac{Q_{\varepsilon}[\varphi]}{\Theta_{\varepsilon}}-\frac{Q[\varphi]}{\Theta}
&=\frac{Q[\varphi]}{\Theta}\frac{\frac{\Theta-\Theta_{\varepsilon}}{\Theta}}{1-\frac{\Theta-\Theta_{\varepsilon}}{\Theta}}+\frac{Q_{\varepsilon}[\varphi]-Q[\varphi]}{\Theta_{\varepsilon}}\\
&=\frac{Q[\varphi]}{\Theta}\frac{\widetilde{M}\rho_{n}(\varepsilon)+E_n(\varepsilon)\rho_{n}(\varepsilon)}{1-\widetilde{M}\rho_{n}(\varepsilon)-E_n(\varepsilon)\rho_{n}(\varepsilon)}+O(\varepsilon^{3/4})\\
&=\frac{Q[\varphi]}{\Theta}\frac{\widetilde{M}\rho_{n}(\varepsilon)}{1-\widetilde{M}\rho_{n}(\varepsilon)}+E_n(\varepsilon)\rho_{n}(\varepsilon).
\end{align*}
For $n=3$, we only need to replace $O(\varepsilon^{3/4})$ by $O(\varepsilon|\log\varepsilon|)$ in the second line of the above equalities. The proof of Proposition \ref{lem_Qbeta} is completed.
\end{proof}

\subsection{Proof of Lemma \ref{lem_24}}

\begin{proof}[Proof of Lemma \ref{lem_24}]
By the definition of ${v}_{1}^{*}$ and ${v}_{2}^{*}$, \eqref{equv1*}, we have
$$
\begin{cases}
\Delta({v}_{1}^{*}+{v}_{2}^{*})=0&\mbox{in}~\widetilde{\Omega}^{*},\\
{v}_{1}^{*}+{v}_{2}^{*}=1&\mbox{on}~\partial{D}_{1}^{*}\cup\partial{D}_{2}^{*},\\
{v}_{1}^{*}+{v}_{2}^{*}=0&\mbox{on}~\partial\Omega.
\end{cases}
$$
By using the Hopf Lemma, we have
$$\frac{\partial({v}_{1}^{*}+{v}_{2}^{*})}{\partial\nu}<0,\quad\mbox{on}~\partial\Omega.$$
Since $0<{v}_{1}^{*}+{v}_{2}^{*}<1$ in $\widetilde{\Omega}^{*}$ and ${v}_{1}^{*}+{v}_{2}^{*}=1$ on $\partial{D}_{1}^{*}\cup\partial{D}_{2}^{*}$, the boundary gradient estimates of a harmonic function implies that there exists a ball $B(\bar{x}, 2\bar{r})\subset\widetilde\Omega$, such that ${v}_{1}^{*}+{v}_{2}^{*}> 1/2$ in $B(\bar{x}, 2\bar{r})$, where $\bar{r}$ is independent of $\varepsilon$. Let $\rho\in{C}^{2}(\overline{\Omega}\setminus\overline{B(\bar{x}, \bar{r})})$ be the solution to
$$
\begin{cases}
\Delta\rho=0 &\mbox{in}~\Omega\setminus\overline{B(\bar{x}, \bar{r})}),\\
\rho=1/2 ~\mbox{on}~\partial{B}(\bar{x}, \bar{r}),& \rho=0~\mbox{on}~\partial\Omega.
\end{cases}
$$
By the maximum principle, $0 <\rho<1/2$ in $\widetilde{\Omega}^{*}\setminus\overline{B(\bar{x}, \bar{r})})$. Using the Hopf Lemma again, 
$$\frac{\partial\rho}{\partial\nu}\leq -\frac{1}{C},\quad\mbox{on}~\partial\Omega.$$
On the other hand, since $\rho\leq{v}_{1}^{*}+{v}_{2}^{*}$ on the boundary of $\widetilde{\Omega}^{*}\setminus\overline{B(\bar{x}, 2\bar{r})})$, it follows from the maximum principle that $0<\rho\leq{v}_{1}^{*}+{v}_{2}^{*}$ in $\widetilde{\Omega}^{*}\setminus\overline{B(\bar{x}, 2\bar{r})})$. In view of $\rho={v}_{1}^{*}+{v}_{2}^{*}=0$ on $\partial\Omega$, 
$$\frac{\partial\rho}{\partial\nu}\geq\frac{\partial({v}_{1}^{*}+{v}_{2}^{*})}{\partial\nu},\quad\mbox{on}~\partial\Omega.$$
Thus,
$$\int_{\partial\Omega}\frac{\partial({v}_{1}^{*}+{v}_{2}^{*})}{\partial\nu}\leq-\frac{1}{C}|\partial\Omega|.$$
This implies that
$$\Theta\geq\frac{1}{C}.$$
Therefore, using $\int_{\partial\Omega}\frac{\partial{v}_{i}}{\partial\nu}\rightarrow\int_{\partial\Omega}\frac{\partial{v}_{i}^{*}}{\partial\nu}$, $i=1,2$, as $\varepsilon\rightarrow0$, see \cite{bly1}, there exists some positive constant $\delta_{0}$ such that $\Theta\geq\delta_{0}$, and $\Theta_{\varepsilon}\geq\delta_{0}/2$ for sufficiently small $\varepsilon$.
\end{proof}

\subsection{Proof of Lemma \ref{lem_25}}

In the following Lemmas for $v_{i}$ and $v_{i}^{*}$, $i=1,2$, we only give the proofs for case $i=1$, since the case $i=2$ is the same.

\begin{lemma}\label{lem_29}
Let $v_{i}$ and $v_{i}^{*}$ be defined as \eqref{equv1} and \eqref{equv1*}, respectively. Then
\begin{equation}\label{v1-v1*}
\|v_{i}-v_{i}^{*}\|_{L^{\infty}\Big(\Omega\setminus{\big((D_{1}\cup{D}_{1}^{*})\cup(D_{2}\cup{D}_{2}^{*})\cup\Omega_{\varepsilon^{1/4}}\big)}\Big)}\leq\,C\varepsilon^{1/2},\qquad\,i=1,2.
\end{equation}
\end{lemma}

\begin{proof}
We will first consider the difference $v_{1}-v_{1}^{*}$ on the boundary of $\Omega\setminus(D_{1}\cup{D}_{2}\cup{D}_{1}^{*}\cup{D}_{2}^{*}\cup\Omega_{\varepsilon^{1/2-\beta}})$, where $0<\beta<1/2$ (small, to be determined later), then use the maximum principle and boundary estimates for elliptic equations to obtain \eqref{v1-v1*}.

{\bf STEP 1.}
First consider the parts on the boundary $\partial(D_{1}\cup{D}_{1}^{*})$. It can be divided into two parts: (a) $\partial{D}_{1}^{*}\setminus{D}_{1}$ and (b) $\partial{D}_{1}\setminus{D}_{1}^{*}$.

(a) For $x\in\partial{D}_{1}^{*}\setminus{D}_{1}$, we introduce a cylinder
$$\mathcal{C}_{r}:=\left\{x\in\mathbb{R}^{n}~\big|~|x'|<r,~-\frac{\varepsilon}{2}+2\min_{|x'|=r}h_{2}(x')\leq\,x_{n}\leq\frac{\varepsilon}{2}+2\max_{|x'|=r}h_{1}(x')\right\},$$
for $r\leq\,R_{0}$. 

(a1) For $x\in\partial{D}_{1}^{*}\cap(\mathcal{C}_{R_{0}}\setminus\mathcal{C}_{\varepsilon^{1/2-\beta}})$, by mean value theorem and estimate \eqref{nabla_v_i0}, we have, for some $\theta_{\varepsilon}\in(0,1)$
\begin{align*}
|v_{1}(x)-v_{1}^{*}(x)|=|v_{1}(x)-1|=&\left|v_{1}(x',h_{1}(x'))-v_{1}(x',\frac{\varepsilon}{2}+h_{1}(x'))\right|\\
=&\left|\partial_{x_{n}}v_{1}(x',\frac{\theta_{\varepsilon}\varepsilon}{2}+h_{1}(x'))\right|\cdot\frac{\varepsilon}{2}\\
\leq&\,\frac{C}{\varepsilon+|x'|^{2}}\cdot\frac{\varepsilon}{2} \leq\,\frac{C\varepsilon}{\varepsilon^{1-2\beta}}=\,C\varepsilon^{2\beta}.
\end{align*}

(a2) For $x\in\partial{D}_{1}^{*}\cap(\widetilde\Omega\setminus\Omega_{R_{0}})$, there exists $y_{\varepsilon}\in\partial{D}_{1}\cap\overline{\widetilde\Omega\setminus\Omega_{R_{0}/2}}$ such that $|x-y_{\varepsilon}|<C\varepsilon$ (note that $v_{1}(y_{\varepsilon})=1$). By \eqref{nabla_v_i_out}, then for some $\theta_{\varepsilon}\in(0,1)$
$$|v_{1}(x)-v_{1}^{*}(x)|=|v_{1}(x)-1|=|v_{1}(x)-v_{1}(y_{\varepsilon})|
\leq|\nabla{v}_{1}((1-\theta_{\varepsilon})x+\theta_{\varepsilon}{y}_{\varepsilon})||x-y_{\varepsilon}|\leq\,C\varepsilon.$$

(b) For $x\in\partial{D}_{1}\setminus{D}_{1}^{*}$, since $0<v_{1}<1$ in $\widetilde{\Omega}$ and $\Delta{v}_{1}=0$ in $\widetilde{\Omega}$, it follows from the boundary estimates of harmonic function that there exists $y_{x}\in\widetilde{\Omega}$, $|y_{x}-x|\leq\,C\varepsilon$ such that $v_{1}(y_{x})=v_{1}^{*}(x)$. Using \eqref{nabla_v_i_out} again, 
$$|v_{1}(x)-v_{1}^{*}(x)|=|v_{1}(x)-v_{1}(y_{x})|\leq\,\|\nabla{v}_{1}\|_{L^{\infty}(\widetilde{\Omega}\setminus\Omega_{R_{0}})}|x-y_{x}|\leq\,C\varepsilon.$$
Therefore,
\begin{equation}\label{onboundaryD1}
|v_{1}(x)-v_{1}^{*}(x)|\leq\,\frac{C\varepsilon}{\varepsilon^{1-2\beta}}=C\varepsilon^{2\beta},\quad\mbox{for}~x\in\partial(D_{1}\cup{D}_{1}^{*})\setminus\mathcal{C}_{\varepsilon^{1/2-\beta}}.
\end{equation}

Similarly, we have
\begin{equation}\label{onboundaryD2}
|v_{1}(x)-v_{1}^{*}(x)|\leq\,C\varepsilon^{2\beta},\quad\mbox{for}~x\in\partial(D_{2}\cup{D}_{2}^{*})\setminus\mathcal{C}_{\varepsilon^{1/2-\beta}}.
\end{equation}

{\bf STEP 2.} Now consider the line segments (or the cylindrical shaped surface in dimension $n=3$) between $\partial{D}_{1}^{*}$ and $\partial{D}_{2}^{*}$, $S_{1/2-\beta}:=\Big\{(x',x_{n})~\big|~|x'|=\varepsilon^{1/2-\beta}, ~h_{2}(x')\leq\,x_{n}\leq\,h_{1}(x')\Big\}$. By using Propostion \ref{prop_v1-u_bar} and the fact that $(v_{1}-\bar{u})=0$ on $\partial{D}_{2}$, we have, for $x\in\,S_{1/2-\beta}$, 
\begin{equation}\label{part1}
|(v_{1}-\bar{u})(x)|\leq\left\|\nabla(v_{1}-\bar{u})\right\|_{L^{\infty}(S_{\beta})}\left|\varepsilon+h_{1}(x')-h_{2}(x')\right|\leq\,C(\varepsilon+|x'|^{2})\leq\,C\varepsilon^{1-2\beta}.
\end{equation}
Similarly, we define $\bar{u}^{*}$, such that $\bar{u}^{*}=1$ on $\partial{D}_{1}^{*}\setminus\{0\}$, $\bar{u}^{*}=0$ on $\partial{D}_{2}^{*}\cup\partial\Omega$,
$$\bar{u}^{*}=\frac{x_{n}-h_{2}(x')}{h_{1}(x')-h_{2}(x')}\qquad\mbox{in}~~\Omega_{R_0}^{*}:=\Big\{(x',x_{n})~\big|~h_{2}(x')\leq\,x_{n}\leq\,h_{1}(x'),~|x'|\leq\,R_0~\Big\},$$
and $\|\bar{u}^{*}\|_{C^{k,1}(\widetilde{\Omega}^{*}\setminus\Omega_{R_0}^{*})}\leq\,C$.
It is easy to see that
$$\bar{u}^{*}=\lim_{\varepsilon\rightarrow0}\bar{u},\quad\mbox{in}~C^{k}(\Omega_{R_0}^{*}\setminus\{0\}),$$
and
\begin{equation}\label{nablau_starbar}
\left|\partial_{x'}\bar{u}^{*}(x)\right|\leq\frac{C}{|x'|},\qquad
\partial_{x_{n}}\bar{u}^{*}(x)=\frac{1}{h_{1}(x')-h_{2}(x')},\quad~x\in\Omega_{R_{0}}^{*}\setminus\{0\}.
\end{equation}
By the proof of Proposition \ref{prop_v1-u_bar}, we also have
\begin{equation}\label{nablaw_0star}
\left\|\nabla(v_{1}^{*}-\bar{u}^{*})\right\|_{L^{\infty}(\widetilde{\Omega}^{*})}\leq\,C.
\end{equation}
Therefore, using $(v_{1}^{*}-\bar{u}^{*})=0$ on $\partial{D}_{2}^{*}$, we have, for $x\in\,S_{1/2-\beta}$, 
\begin{equation}\label{part2}
|(v_{1}^{*}-\bar{u}^{*})(x)|\leq\left\|\nabla(v_{1}^{*}-\bar{u}^{*})\right\|_{L^{\infty}(S_{1/2-\beta})}\left|h_{1}(x')-h_{2}(x')\right|\leq\,C|x'|^{2}\leq\,C\varepsilon^{1-2\beta}.
\end{equation}
Finally, by the definitions of $\bar{u}$ and $\bar{u}^{*}$, for $x\in\,S_{1/2-\beta}$, 
\begin{align}\label{part3}
&|(\bar{u}-\bar{u}^{*})(x)|\nonumber\\
&\leq\left|\bar{u}(x',h_{2}(x'))-\bar{u}(x',-\frac{\varepsilon}{2}+h_{2}(x'))\right|+\left\|\partial_{x_{n}}(\bar{u}-\bar{u}^{*})\right\|_{L^{\infty}(S_{1/2-\beta})}\left|h_{1}(x')-h_{2}(x')\right|\nonumber\\
&\leq\,\left|\partial_{x_{n}}\bar{u}(x',-\frac{\theta_{\varepsilon}\varepsilon}{2}+h_{2}(x'))\right|\cdot\frac{\varepsilon}{2}+C\left(\frac{1}{h_{1}(x')-h_{2}(x')}-\frac{1}{\varepsilon+h_{1}(x')-h_{2}(x')}\right)|x'|^{2}\nonumber\\
&\leq\frac{C\varepsilon}{\varepsilon+|x'|^{2}}+\frac{C\varepsilon}{|x'|^{2}(\varepsilon+|x'|^{2})}|x'|^{2} \leq\,C\varepsilon^{2\beta}.
\end{align}
Taking $\beta=1/4$, by \eqref{part1}, \eqref{part2} and \eqref{part3}, we have, for $x\in\,S_{1/4}$, 
\begin{equation*}%\label{onboundaryS}
|(v_{1}-v_{1}^{*})(x)|\leq|(v_{1}-\bar{u})(x)|+|(\bar{u}-\bar{u}^{*})(x)|+|(\bar{u}^{*}-v_{1}^{*})(x)|\leq\,C\varepsilon^{1/2}.
\end{equation*}

Combining with \eqref{onboundaryD1}, \eqref{onboundaryD2} for $\beta=1/4$, recalling $v_{1}-v_{1}^{*}\equiv0$ on $\partial\Omega$, and using maximum principle, we obtain
\eqref{v1-v1*}.
\end{proof}

Outside of $\Omega_{R_{0}}$, we have the following improvement of Lemma \ref{lem_29}. 
\begin{lemma}\label{lem_2.10}
Let $v_{i}$ and $v_{i}^{*}$ be defined as \eqref{equv1} and \eqref{equv1*}, respectively. Then
\begin{equation}\label{v1-v1*2}
\|v_{i}-v_{i}^{*}\|_{L^{\infty}\Big(\Omega\setminus{\big(D_{1}\cup{D}_{1}^{*}\cup{D}_{2}\cup{D}_{2}^{*}\cup\Omega_{R_{0}}\big)}\Big)}\leq\,\begin{cases}
C\varepsilon^{3/4}, &\mbox{if}~n=2;\\
C\varepsilon|\log\varepsilon|, &\mbox{if}~n=3,
\end{cases}\qquad\,i=1,2.
\end{equation}
\end{lemma}

\begin{proof}

Let $k_{1}$ be $2^{-k_{1}-1}\leq\varepsilon^{1/4}\leq 2^{-k_{1}}$, and $k_{0}$ be $2^{-k_{0}-1}\leq R_{0}/2\leq2^{-k_{0}}$, since $R_{0}<1$. Since for sufficiently small $\varepsilon$, $\partial(D_{i}\cup{D}_{i}^{*})\cap\mathcal{C}_{R_{0}/2}=\partial{D}_{i}^{*}\cap\mathcal{C}_{R_{0}/2}$, we denote
$$E^{k}_{i}:=\left(\mathcal{C}_{2^{-k}}\setminus\mathcal{C}_{2^{-k-1}}\right)\cap\partial{D}_{i}^{*},\quad\mbox{for}~~k_{0}\leq k\leq k_{1},\quad\,i=1,2.$$
Then
$$\cup_{k=k_{0}}^{k_{1}}E^{k}_{i}=\left(\mathcal{C}_{R_{0}/2}\setminus\mathcal{C}_{\varepsilon^{1/4}}\right)\cap\partial{D}_{i}^{*},\quad\,i=1,2.$$

It follows from \eqref{onboundaryD1} that  
\begin{equation}\label{v1-v1*_Ek}
|v_{1}(x)-v_{1}^{*}(x)|\leq\,C\varepsilon\cdot2^{2k},\quad\,x\in\,E^{k}_{i},\quad\mbox{for}~~k_{0}\leq k\leq k_{1},\quad\,i=1,2.
\end{equation}
For each $E^{k}_{i}$, we will construct a positive harmonic function $\xi^{k}_{i}$ as below. We will use 
$$\xi_{i}:=\sum_{k=k_{0}}^{k_{1}}\xi^{k}_{i}$$
as one of the few harmonic functions to bound $\pm(v_{1}-v_{1}^{*})$ on $\partial(D_{i}\cup{D}_{i}^{*})$ from above in the following. Let $\tilde{\xi}^{k}_{i} $ be the solution of
$$
\begin{cases}
\Delta\tilde{\xi}^{k}_{i} =0,~\quad\mbox{in}~\mathbb{R}^{n}\setminus\,D_{i}^{*},\\
\tilde{\xi}^{k}_{i} =1,~\mbox{on}~E^{k}_{i},\quad\tilde{\xi}^{k} _{i} =0,~\mbox{on}~\partial{D}_{i}^{*}\setminus\,E^{k} _{i},\\
\tilde{\xi}^{k}_{i} \in L^\infty(\mathbb{R}^n \setminus D_{i}^*), \quad ~\mbox{if}~ n =2,\\
\tilde{\xi}^{k}_{i} \to 0 ~~\mbox{as}~~ |x| \to \infty,  \quad ~\mbox{if}~ n =3.
\end{cases}
$$
By the representation formula for the solution of the above boundary value problem using Green's function, we have
$$\tilde{\xi}^{k}_{i}(x)=\int_{E^{k}_{i}}\frac{\partial G_{i}}{\partial\nu^{-}}(x,y)dS(y),$$
 where $G_{i}(x,y)$ is the Green's function for the domain $\mathbb{R}^{n}\setminus{D}_{i}$ which satisfies
\begin{equation}\label{Green_bound}
|\nabla_{y} G_{i}(x,y)|\leq\,C\,\quad\mbox{if}~~y\in E^{k}_{i}\quad\mbox{and}~~x\in\widetilde{\Omega}^{*}\cap\partial\mathcal{C}_{R_{0}},\quad~\forall~k_{0}\leq\,k\leq\,k_{1}.
\end{equation}
In view of \eqref{v1-v1*_Ek}, we take
$$\xi^{k}_{i}:=C\varepsilon\cdot2^{2k}\tilde{\xi}^{k}_{i}(x)$$
and use
$$\xi_{i}(x)=\sum_{k=k_{0}}^{k_{1}}\xi^{k}_{i},\quad i=1,2,$$
to bound $\pm(v_{1}-v_{1}^{*})$ on $\left(\mathcal{C}_{R_{0}/2}\setminus\mathcal{C}_{\varepsilon^{1/4}}\right)\cap\partial{D}_{i}^{*}$ from above.

Let ${B}_{r}$ denote the ball of radius $r$ centered at the origin in $\mathbb{R}^{n}$. Now due to \eqref{v1-v1*}, $|(v_1- v^*_1)(x)| \le C \varepsilon^{1/2}$ at $\widetilde{\Omega}^{*} \cap \partial{B}_{2\varepsilon^{1/4}}$, we will construct another auxiliary function to bound $\pm(v_{1}-v_{1}^{*})$ on $\widetilde{\Omega}^{*} \cap \partial{B}_{2\varepsilon^{1/4}}$ from above. Define $$\Sigma_{\delta}:=\left\{(x_{1},x_{2},\cdots,x_{n}) \in \partial{B}_{1} ~\big| |x_{n}|<\delta\right\}.$$
Let $u_\delta$ be the solution of
$$
\begin{cases}
\Delta{u_\delta}=0,&\mbox{in}~~B_{1}\subset\mathbb{R}^{n},\\
u_\delta(x)=1 ,&\mbox{on}~~\Sigma_{C_0 \delta},\\
u_\delta(x)=0 ,&\mbox{on}~~\partial{B}_{1}\setminus \Sigma_{C_0 \delta},
\end{cases}
$$
where $C_0$ is a constant such that $\sum_{i=1}^n \lambda_i \le C_0$. From the Green's representation, we have
$$u_\delta(x)=\frac{1-|x|^{2}}{n\alpha(n)}\int_{\Sigma_{C_0 \delta}}\frac{1}{|x-y|^{n-2}}dS_{y},$$
where $\alpha(n)=|B_{1}|$. Then for $|x|\leq\frac{3}{4}$,
$$0<u_\delta(x)\leq\,C\int_{\partial{B}_{1}\cap\{|x_{n}|<C_0\delta\}}dS_{y}\leq\,C\delta.$$
By the Kelvin transformation, let
$$\tilde{u}_\delta(x):=\frac{1}{|x|^{n-2}}u_\delta\left(\frac{x}{|x|^{2}}\right),\quad\mbox{for}~~|x|>1,$$
then $\tilde{u}_\delta(x) = 1$ on $\Sigma_{C_0 \delta}$,
$$\Delta \tilde{u}_\delta(x)=0,\quad \tilde{u}_\delta(x)>0,\quad\mbox{for}~~|x|>1,$$
and as $|x|\rightarrow\infty$, $\tilde{u}_\delta(x)\rightarrow\,u_\delta(0)=\frac{|\Sigma_{C_0\delta}|}{|\partial {B}_{1}|}$  if $n=2$; $\tilde{u}_\delta(x)\rightarrow0$ for $n\geq3$. Furthermore, for $|x|\geq\frac{4}{3}$, we have
\begin{equation}\label{udeltadecay}
\tilde{u}_\delta(x)\leq\frac{C\delta}{|x|^{n-2}}.
\end{equation}
Take
$$\xi_0: = \bar{C}\varepsilon^{1/2}\tilde{u}_\delta(x)\left(\frac{x}{\delta} \right)$$
with $\delta = 2 \varepsilon^{1/4}$, where $\bar{C}$ is the same constant $C$ in \eqref{v1-v1*}. Because of the choice of $C_0$ and \eqref{v1-v1*}, we can see
$$\xi_0 = \bar{C} \varepsilon^{1/2} \geq\, \pm(v_{1}-v_{1}^{*}) \quad \mbox{on}~~\widetilde{\Omega}^{*} \cap \partial{B}_{2\varepsilon^{1/4}}.$$
And according to \eqref{udeltadecay}
$$ \xi_0 \le C \varepsilon^{\frac{n+1}{4}} \quad  \mbox{on}~~\widetilde{\Omega}^{*} \cap \partial C_{R_0}.$$

Due to $\|\nabla{v}_{1}\|_{L^{\infty}(\widetilde{\Omega}\setminus\Omega_{R/2})}\leq\,C$ and $\|\nabla{v}_{1}^{*}\|_{L^{\infty}(\widetilde{\Omega}^{*}\setminus\Omega_{R/2}^{*})}\leq\,C$, we have 
$$\pm(v_{1}-v_{1}^{*})\leq\,C\varepsilon,\qquad\mbox{on}~~\partial(D_{1}\cup{D}_{1}^{*}\cup{D}_{2}\cup{D}_{2}^{*})\setminus\mathcal{C}_{R_{0}/2}.$$
In view of $v_{1}-v_{1}^{*}=0$ on $\partial\Omega$ and the positivity of $\xi_{i}$, $i=0,1,2$, we have 
$$\pm(v_{1}-v_{1}^{*})\leq\xi_{0}+\xi_{1}+\xi_{2}+C\varepsilon,\quad\mbox{on}~~\partial\Big(\Omega\setminus{\big(D_{1}\cup{D}_{1}^{*}\cup{D}_{2}\cup{D}_{2}^{*}\cup B_{2\varepsilon^{1/4}}\big)}\Big).
$$
By using the maximum principle in $\Omega\setminus{\big(D_{1}\cup{D}_{1}^{*}\cup{D}_{2}\cup{D}_{2}^{*}\cup B_{2\varepsilon^{1/4}}\big)}$, we have
\begin{equation}\label{max_principle}
\pm(v_{1}-v_{1}^{*})\leq\xi_{0}+\xi_{1}+\xi_{2}+C\varepsilon,\quad\mbox{in}~~\Omega\setminus{\big(D_{1}\cup{D}_{1}^{*}\cup{D}_{2}\cup{D}_{2}^{*}\cup B_{2\varepsilon^{1/4}}\big)}.
\end{equation}

Next, in order to prove \eqref{v1-v1*2}, we need to further estimate $\xi_{i}$ on $\widetilde{\Omega}^{*}\cap\partial\mathcal{C}_{R_{0}}$, $i=1,2$. Making use of \eqref{Green_bound}, 
\begin{equation*}%\label{xi_k^i}
\tilde{\xi}^{k}_{i}(x)\leq\,C|E^{k}_{i}|\leq\frac{C}{2^{(n-1)k}},\qquad~x\in\widetilde{\Omega}^{*}\cap\partial\mathcal{C}_{R_{0}}.
\end{equation*}
Thus
$$\xi_{i}(x)\leq\,C\sum_{k=k_{0}}^{k_{1}}\varepsilon\cdot2^{2k}\frac{C}{2^{(n-1)k}}=C\sum_{k=k_{0}}^{k_{1}}\frac{\varepsilon}{2^{(n-3)k}},\qquad~x\in\widetilde{\Omega}^{*}\cap\partial\mathcal{C}_{R_{0}}.$$
Hence, if $n=2$, recalling $k_{1}\sim\frac{1}{4\log 2}|\log\varepsilon|$, 

$$\xi_{i}(x)\leq\,C\sum_{k=k_{0}}^{k_{1}}\varepsilon2^{k}\leq\,C\varepsilon2^{k_{1}}\leq\,C\varepsilon^{3/4},\quad\,x\in\widetilde{\Omega}^{*}\cap\partial\mathcal{C}_{R_{0}};$$

if $n=3$, 

$$\xi_{i}(x)\leq\,C\varepsilon k_{1}\leq\,C\varepsilon|\log\varepsilon|,\quad\,x\in\widetilde{\Omega}^{*}\cap\partial\mathcal{C}_{R_{0}}.$$

Combining these estimates above with \eqref{max_principle}, we have, on $\partial\Big(\Omega\setminus{\big(D_{1}\cup{D}_{1}^{*}\cup{D}_{2}\cup{D}_{2}^{*}\cup\Omega_{R_{0}}\big)}\Big)$, 
$$\pm(v_{1}-v_{1}^{*})\leq\xi_{0}+\xi_{1}+\xi_{2}+C\varepsilon\leq\,\begin{cases}
C\varepsilon^{3/4} +C\varepsilon^{3/4}+C\varepsilon, &\mbox{if}~n=2;\\
C\varepsilon+C\varepsilon|\log\varepsilon|+C\varepsilon, &\mbox{if}~n=3.
\end{cases}$$
By using the maximum principle again, 
$$|v_{1}-v_{1}^{*}|\leq\,\begin{cases}
C\varepsilon^{3/4}, &\mbox{if}~n=2;\\
C\varepsilon|\log\varepsilon|, &\mbox{if}~n=3,
\end{cases}\quad\mbox{in}~~\Omega\setminus{\big(D_{1}\cup{D}_{1}^{*}\cup{D}_{2}\cup{D}_{2}^{*}\cup\Omega_{R_{0}}\big)}.$$
The proof is completed.
\end{proof}

An immediate consequence of Lemma \ref{lem_2.10} and the boundary estimates for elliptic equations is as follows:

\begin{lemma}\label{lem_beta_difference}
Let $v_{i}$ and $v_{i}^{*}$ be defined as \eqref{equv1} and \eqref{equv1*}, respectively. Then
\begin{equation}\label{b1+b2}
\left|\int_{\partial{\Omega}}\frac{\partial{v}_{i}}{\partial\nu}
-\int_{\partial{\Omega}}\frac{\partial{v}_{i}^{*}}{\partial\nu}\right|\leq\,\begin{cases}
C\varepsilon^{3/4}, &\mbox{if}~n=2;\\
C\varepsilon|\log\varepsilon|, &\mbox{if}~n=3,
\end{cases}\quad\,i=1,2.
\end{equation}
\end{lemma}

Now we prove Lemma \ref{lem_25}.

\begin{proof}[Proof of Lemma \ref{lem_25}]
Since
$$\int_{\partial{D}_{1}}\frac{\partial{v}_{1}}{\partial\nu^{-}}
=\int_{\widetilde{\Omega}}|\nabla{v}_{1}|^{2},$$
it follows from Theorem \ref{thm_energy} that
\begin{equation*}
\rho_{n}(\varepsilon)\int_{\widetilde{\Omega}}|\nabla{v}_{i}|^{2}
-\kappa_n=M_{i}\rho_{n}(\varepsilon)+E_n(\varepsilon)\rho_{n}(\varepsilon),\quad\,i=1,2.
\end{equation*}
In view of the definitions of $v_{1}$ and $v_{2}$ and the Green's formula, we obtain the following identity
$$\int_{\partial{D}_{1}}\frac{\partial{v}_{2}}{\partial\nu^{-}}=
\int_{\partial{D}_{2}}\frac{\partial{v}_{1}}{\partial\nu^{-}}=
-\int_{\partial{D}_{1}}\frac{\partial{v}_{1}}{\partial\nu^{-}}-
\int_{\partial{\Omega}}\frac{\partial{v}_{1}}{\partial\nu}.$$
Thus, recalling the definition of $\Theta_{\varepsilon}$, \eqref{def_beta},
\begin{align*}
\Theta_{\varepsilon}&=-\left(\rho_{n}(\varepsilon)\int_{\partial{D}_{1}}\frac{\partial{v}_{1}}{\partial\nu^{-}}\right)\int_{\partial{\Omega}}\frac{\partial{v}_{2}}{\partial\nu}
+\left(\rho_{n}(\varepsilon)\int_{\partial{D}_{1}}\frac{\partial{v}_{2}}{\partial\nu^{-}}\right)\int_{\partial{\Omega}}\frac{\partial{v}_{1}}{\partial\nu}\\
&=-\left(\rho_{n}(\varepsilon)\int_{\partial{D}_{1}}\frac{\partial{v}_{1}}{\partial\nu^{-}}\right)\int_{\partial{\Omega}}\frac{\partial({v}_{1}+{v}_{2})}{\partial\nu}-\rho_{n}(\varepsilon)\left(\int_{\partial{\Omega}}\frac{\partial{v}_{1}}{\partial\nu}\right)^{2}.
\end{align*}
Recalling the definition of $\Theta$, \eqref{def_beta*},
$$\Theta=-\kappa_n\int_{\partial{\Omega}}\frac{\partial({v}_{1}^{*}+ {v}_{2}^{*})}{\partial\nu}$$
and using Lemma \ref{lem_beta_difference} and Theorem \ref{thm_energy}, we have,
\begin{align*}
\Theta-\Theta_{\varepsilon}=&
\left(\rho_{n}(\varepsilon)\int_{\partial{D}_{1}}\frac{\partial{v}_{1}}{\partial\nu^{-}}-\kappa_n\right)\int_{\partial{\Omega}}\frac{\partial({v}_{1}^{*}+{v}_{2}^{*})}{\partial\nu}
+\rho_{n}(\varepsilon)\left(\int_{\partial{\Omega}}\frac{\partial{v}_{1}^{*}}{\partial\nu}\right)^{2}\\
&\hspace{2cm}+O(\varepsilon^{\frac{3}{4}})(\mbox{or}~O(\varepsilon|\log\varepsilon|))\\
=&\left(M_{1}\rho_{n}(\varepsilon)+E_n(\varepsilon)\rho_{n}(\varepsilon)\right)\int_{\partial{\Omega}}\frac{\partial(v_{1}^{*}+{v}_{2}^{*})}{\partial\nu}+\rho_{n}(\varepsilon)\left(\int_{\partial{\Omega}}\frac{\partial{v}_{1}^{*}}{\partial\nu}\right)^{2}\\
&\hspace{2cm}+O(\varepsilon^{\frac{3}{4}})(\mbox{or}~O(\varepsilon|\log\varepsilon|))\\
=&\rho_{n}(\varepsilon)\left(M_{1}\int_{\partial{\Omega}}\frac{\partial(v_{1}^{*}+{v}_{2}^{*})}{\partial\nu}+
\Big(\int_{\partial{\Omega}}\frac{\partial{v}_{1}^{*}}{\partial\nu}\Big)^{2}\right)+E_n(\varepsilon)\rho_{n}(\varepsilon).
\end{align*}
\eqref{beta_difference} is proved.
\end{proof}

\subsection{Proof of Lemma \ref{lem_26}}

To prove \eqref{Q_difference}, besides \eqref{b1+b2}, we need

\begin{lemma}\label{lem_beta_difference2}
Let $v_{0}$ and $v_{0}^{*}$ be defined in \eqref{equv1} and \eqref{equv1*}, respectively. Then
\begin{equation}\label{bi}
\left|\int_{\partial{D}_{i}}\frac{\partial{v}_{0}}{\partial\nu^{-}}
-\int_{\partial{D}_{i}^{*}}\frac{\partial{v}_{0}^{*}}{\partial\nu^{-}}\right|\leq\,C\|\varphi\|_{L^{\infty}(\partial\Omega)}\begin{cases}
\varepsilon^{3/4}, &\mbox{if}~n=2,\\
\varepsilon|\log\varepsilon|, &\mbox{if}~n=3,
\end{cases}\quad\,i=1,2.
\end{equation}
\end{lemma}

\begin{proof}
Using the Green's formula,
$$\int_{\partial{D}_{1}}\frac{\partial{v}_{0}}{\partial\nu^{-}}=\int_{\partial{D}_{1}}\frac{\partial{v}_{0}}{\partial\nu^{-}}v_{1}= \int_{\partial\widetilde{\Omega}}\frac{\partial{v}_{0}}{\partial\nu}v_{1}=\int_{\partial\widetilde{\Omega}}\frac{\partial{v}_{1}}{\partial\nu}v_{0} =\int_{\partial{\Omega}}\frac{\partial{v}_{1}}{\partial\nu}\varphi,$$
and
$$\int_{\partial{D}_{1}^{*}}\frac{\partial{v}_{0}^{*}}{\partial\nu^{-}}=\int_{\partial{D}_{1}^{*}}\frac{\partial{v}_{0}^{*}}{\partial\nu^{-}}v_{1}^{*}=\int_{\partial{\Omega}}\frac{\partial{v}_{1}^{*}}{\partial\nu}\varphi.$$
So that, by Lemma \ref{lem_beta_difference},
$$\left|\int_{\partial{D}_{1}}\frac{\partial{v}_{0}}{\partial\nu^{-}}
-\int_{\partial{D}_{1}^{*}}\frac{\partial{v}_{0}^{*}}{\partial\nu^{-}}\right|=
\left|\int_{\partial{\Omega}}\frac{\partial({v}_{1}-{v}_{1}^{*})}{\partial\nu}\varphi\right|\leq\,\begin{cases}
C\varepsilon^{3/4}, &\mbox{if}~n=2,\\
C\varepsilon|\log\varepsilon|, &\mbox{if}~n=3.
\end{cases}$$
This completes the proof, with the assumption $\|\varphi\|_{C^{2}(\partial \Omega)}=1$.
\end{proof}

\begin{proof}[Proof of Lemma \ref{lem_26}]
Recall that
$$Q_{\varepsilon}[\varphi]=\int_{\partial{D}_{1}}\frac{\partial{v}_{0}}{\partial\nu^{-}}
\int_{\partial{\Omega}}\frac{\partial{v}_{2}}{\partial\nu}
-\int_{\partial{D}_{2}}\frac{\partial{v}_{0}}{\partial\nu^{-}}\int_{\partial{\Omega}}\frac{\partial{v}_{1}}{\partial\nu},
$$
and
$$Q[\varphi]=\int_{\partial{D}_{1}^{*}}\frac{\partial{v}_{0}^{*}}{\partial\nu^{-}}
\int_{\partial{\Omega}}\frac{\partial{v}_{2}^{*}}{\partial\nu}
-\int_{\partial{D}_{2}^{*}}\frac{\partial{v}_{0}^{*}}{\partial\nu^{-}}\int_{\partial{\Omega}}\frac{\partial{v}_{1}^{*}}{\partial\nu}.
$$
Using \eqref{b1+b2} and \eqref{bi}, we have
\begin{align*}
&|Q_{\varepsilon}[\varphi]-Q[\varphi]|\\
\leq&\left|\left(\int_{\partial{D}_{1}}\frac{\partial{v}_{0}}{\partial\nu^{-}}
-\int_{\partial{D}_{1}^{*}}\frac{\partial{v}_{0}^{*}}{\partial\nu^{-}}\right)
\int_{\partial{\Omega}}\frac{\partial{v}_{2}}{\partial\nu}\right|+
\left|\int_{\partial{D}_{1}^{*}}\frac{\partial{v}_{0}^{*}}{\partial\nu^{-}}\left(\int_{\partial{\Omega}}\frac{\partial{v}_{2}}{\partial\nu}
-\int_{\partial{\Omega}}\frac{\partial{v}_{2}^{*}}{\partial\nu}\right)\right|\\
&+\left|\left(\int_{\partial{D}_{2}}\frac{\partial{v}_{0}}{\partial\nu^{-}}
-\int_{\partial{D}_{2}^{*}}\frac{\partial{v}_{0}^{*}}{\partial\nu^{-}}\right)
\int_{\partial{\Omega}}\frac{\partial{v}_{1}}{\partial\nu}\right|+
\left|\int_{\partial{D}_{2}^{*}}\frac{\partial{v}_{0}^{*}}{\partial\nu^{-}}\left(\int_{\partial{\Omega}}\frac{\partial{v}_{1}}{\partial\nu}
-\int_{\partial{\Omega}}\frac{\partial{v}_{1}^{*}}{\partial\nu}\right)\right|\\
\leq&C\|\varphi\|_{L^{\infty}(\partial\Omega)}\begin{cases}
\varepsilon^{3/4}, &\mbox{if}~n=2,\\
\varepsilon|\log\varepsilon|, &\mbox{if}~n=3.
\end{cases}
\end{align*}
So \eqref{Q_difference} is proved.
\end{proof}

\section{Proof of Theorem \ref{thm_energy}}\label{sec_thm_energy}

Using Lemma \ref{lem_29}, we have
\begin{lemma}\label{lem_49}
Assume that $v_{1}$ and $v_{1}^{*}$ are solutions of \eqref{equv1} and \eqref{equv1*}, respectively. If $\partial{D}_{1}^{*}$ and $\partial{D}_{2}^{*}$ are of $C^{k,1}$, $k\geq 3$, then for $\varepsilon^{1/4}\leq|x'|\leq{R}_{0}$, we have
\begin{equation}\label{nabla_v1v1*_inunitsize}
|\nabla{v}_{1}(x)|\leq\,C|x'|^{-2},\quad\,x\in\Omega_{R_0}\setminus\Omega_{\varepsilon^{1/4}},
\quad|\nabla{v}^{*}_{1}(x)|\leq\,C|x'|^{-2},\quad\,x\in\Omega^{*}_{R_0}\setminus\Omega^{*}_{\varepsilon^{1/4}};
\end{equation}
and
\begin{equation}\label{nabla_v1-v1*_inunitsize}
|\nabla(v_{1}-v_{1}^{*})(x)|\leq\,C\varepsilon^{1/2(1-\frac{1}{k})}|x'|^{-2},\quad\mbox{in}~~\Omega^{*}_{R_0}\setminus\Omega^{*}_{\varepsilon^{1/4}}.
\end{equation}
\end{lemma}

\begin{proof}
For $\varepsilon^{1/4}\leq|z'|\leq\,R_0$, use the change of variable as before
\begin{equation*}%\label{changeofvariable2}
 \left\{
  \begin{aligned}
  &x'-z'=|z'|^{2} y',\\
  &x_n=|z'|^{2} y_n,
  \end{aligned}
 \right.
\end{equation*}
to rescale $\Omega_{|z'|+ |z'|^{2}}\setminus\Omega_{|z'|}$ into a nearly unit-size square (or cylinder) $Q_{1}$, and $\Omega^{*}_{|z'|+|z'|^{2}}\setminus\Omega^{*}_{|z'|}$ into $Q_{1}^{*}$. Let
$$V_{1}(y)=v_{1}(z'+|z'|^{2}y',|z'|^{2}y_{n}),\quad\mbox{in}~~Q_{1},$$
and
$$V_{1}^{*}(y)=v_{1}^{*}(z'+|z'|^{2}y',|z'|^{2}y_{n}),\quad\mbox{in}~~Q_{1}^{*}.$$
Since $0<V_{1},V_{1}^{*}<1$, using the standard elliptic estimate, we have
\begin{equation*}%\label{resacaled_v1}
|\nabla^{k}V_{1}|\leq\,C(k),\quad\mbox{in}~~Q_{1},\quad
\mbox{and}~~
|\nabla^{k}V_{1}^{*}|\leq\,C(k),\quad\mbox{in}~~Q_{1}^{*}.
\end{equation*}

By using an interpolation with \eqref{v1-v1*}, we have
$$|\nabla(V_{1}-V_{1}^{*})|\leq\,C(k)\varepsilon^{1/2(1-\frac{1}{k})},\quad\mbox{in}~~Q_{1}^{*}.$$
Thus, back to $v_{1}-v_{1}^{*}$, we have
\begin{equation*}%\label{nabla_v1-v1*_inunitsize}
|\nabla(v_{1}-v_{1}^{*})(x)|\leq\,C\varepsilon^{1/2(1-\frac{1}{k})}|z'|^{-2},\quad\,x\in\Omega^{*}_{|z'|+|z'|^{2}}\setminus\Omega^{*}_{|z'|}.
\end{equation*}
\eqref{nabla_v1-v1*_inunitsize} follows. By the way, 
\begin{equation*}%\label{nabla_v1v1*_inunitsize}
|\nabla{v}_{1}(x)|\leq\,C|z'|^{-2},\quad\,x\in\Omega_{|z'|+|z'|^{2}}\setminus\Omega_{|z'|},
\quad|\nabla{v}^{*}_{1}(x)|\leq\,C|z'|^{-2},\quad\,x\in\Omega^{*}_{|z'|+|z'|^{2}}\setminus\Omega^{*}_{|z'|}.
\end{equation*}
So \eqref{nabla_v1v1*_inunitsize} follows.
\end{proof}

\begin{proof}[Proof of Theorem \ref{thm_energy}]
We only prove the case for $i=1$ that
\begin{equation}\label{a1i2}
\rho_{n}(\varepsilon)\int_{\widetilde{\Omega}}|\nabla{v}_{1}|^{2}
-\kappa_n=M_{1}\rho_{n}(\varepsilon)+
\left\{ \begin{aligned}
&O \left(\varepsilon^{\frac{1}{4}-\frac{1}{2k}} \right)\rho_{n}(\varepsilon), && \mbox{if}~n = 2,\\
&O\left(\varepsilon^{\frac{1}{2}-\frac{1}{2k}}|\log \varepsilon|\right)\rho_{n}(\varepsilon), &&\mbox{if}~ n = 3.
\end{aligned}\right.
\end{equation}
The case for $i=2$ is the same.

\noindent{\bf STEP 1.} For $0<\gamma\leq1/4$, we divide the integral into three parts:
\begin{equation*}%\label{energy_v1}
\int_{\widetilde{\Omega}}|\nabla{v}_{1}|^{2}
=\int_{\Omega_{\varepsilon^{\gamma}}}|\nabla{v}_{1}|^{2}
+\int_{\Omega_{R_{0}}\setminus\Omega_{\varepsilon^{\gamma}}}|\nabla{v}_{1}|^{2}+\int_{\widetilde{\Omega}\setminus\Omega_{R_{0}}}|\nabla{v}_{1}|^{2}
=:\mathrm{I}+\mathrm{II}+\mathrm{III}.
\end{equation*}

(i) For the first term $I$,
\begin{align}\label{nabla_v1_part3}
\int_{\Omega_{\varepsilon^{\gamma}}}|\nabla{v}_{1}|^{2}=
&\int_{\Omega_{\varepsilon^{\gamma}}}|\nabla\bar{u}|^{2}
+2\int_{\Omega_{\varepsilon^{\gamma}}}\nabla\bar{u}\cdot\nabla(v_{1}-\bar{u})
+\int_{\Omega_{\varepsilon^{\gamma}}}|\nabla(v_{1}-\bar{u})|^{2}.
\end{align}
Recalling \eqref{nablau_bar_inside},we have
$$\int_{\Omega_{\varepsilon^{\gamma}}}|\partial_{x'}\bar{u}|^{2} \le C \int_{|x'|<\varepsilon^\gamma} \frac{|x'|^2}{\varepsilon + |x'|^2} \,dx'  \le C \int_{|x'|<\varepsilon^\gamma} \,dx' = O\left(\varepsilon^{(n-1)\gamma}\right).$$
By combining Proposition \ref{prop_v1-u_bar}, we have
$$2\int_{\Omega_{\varepsilon^{\gamma}}}\nabla\bar{u}\cdot\nabla(v_{1}-\bar{u})
+\int_{\Omega_{\varepsilon^{\gamma}}}|\nabla(v_{1}-\bar{u})|^{2}=O\left(\varepsilon^{(n-1)\gamma}\right).
$$
Hence, it follows from \eqref{nabla_v1_part3} that
\begin{align*}
\mathrm{I}=\int_{\Omega_{\varepsilon^{\gamma}}}|\nabla{v}_{1}|^{2}
= \int_{\Omega_{\varepsilon^{\gamma}}}|\partial_{x_{n}}\bar{u}|^{2} + O\left(\varepsilon^{(n-1)\gamma}\right)
=&\int_{|x'|<\varepsilon^{\gamma}}\frac{dx'}{\varepsilon+h_{1}(x')-h_{2}(x')}+O\left(\varepsilon^{(n-1)\gamma}\right).
\end{align*}

(ii) For the second term, we divide it further as follows:
\begin{align*}
\mathrm{II}=\int_{\Omega_{R_{0}}\setminus\Omega_{\varepsilon^{\gamma}}}|\nabla{v}_{1}|^{2}
=&\int_{(\Omega_{R_{0}}\setminus\Omega_{\varepsilon^{\gamma}})\setminus(\Omega^{*}_{R_0}\setminus\Omega^{*}_{\varepsilon^{\gamma}})}|\nabla{v}_{1}|^{2}+\int_{\Omega^{*}_{R_0}\setminus\Omega^{*}_{\varepsilon^{\gamma}}}|\nabla(v_{1}-v^{*}_{1})|^{2}\\
&+2\int_{\Omega^{*}_{R_0}\setminus\Omega^{*}_{\varepsilon^{\gamma}}}\nabla{v}^{*}_{1}\cdot\nabla(v_{1}-v^{*}_{1})+\int_{\Omega^{*}_{R_0}\setminus\Omega^{*}_{\varepsilon^{\gamma}}}|\nabla{v}^{*}_{1}|^{2}
\\
=:&\mathrm{II}_{1}+\mathrm{II}_{2}+\mathrm{II}_{3}+\mathrm{II}_{4}.
\end{align*}
Noting that the thickness of $(\Omega_{R_{0}}\setminus\Omega_{\varepsilon^{\gamma}})\setminus(\Omega^{*}_{R_0}\setminus\Omega^{*}_{\varepsilon^{\gamma}})$ is $\varepsilon$, and using Lemma \ref{lem_49},
$$\mathrm{II}_{1}=\int_{(\Omega_{R_{0}}\setminus\Omega_{\varepsilon^{\gamma}})\setminus(\Omega^{*}_{R_0}\setminus\Omega^{*}_{\varepsilon^{\gamma}})}|\nabla{v}_{1}|^{2}
\leq\,C\varepsilon\int_{\varepsilon^{\gamma}<|x'|<R_0}\frac{dx'}{|x'|^{4}}\leq\,C\varepsilon^{1+(n-5)\gamma}.$$
For any $\varepsilon^{\gamma}\leq|z'|\leq{R}_{0}$, $0<\gamma\leq1/4$,
by Lemma \ref{lem_49}, if $\partial{D}_{1}^{*}$ and $\partial{D}_{2}^{*}$ are of $C^{k,1}$, $k\geq 3$, then we have
\begin{align*}
\mathrm{II}_{2}=\int_{\Omega^{*}_{R_0}\setminus\Omega^{*}_{\varepsilon^{\gamma}}}|\nabla(v_{1}-v_{1}^{*})|^{2}\leq&\,C\varepsilon^{1-\frac{1}{k}}\int_{\Omega^{*}_{R_0}\setminus\Omega^{*}_{\varepsilon^{\gamma}}}|x'|^{-4}dx'dx_{n}\\
\leq&\,C\varepsilon^{1-\frac{1}{k}}\int_{\varepsilon^{\gamma}<|x'|<R_0}\frac{dx'}{|x'|^{2}}\\
\leq&\begin{cases}
C\varepsilon^{1-\frac{1}{k}-\gamma},&\mbox{if}~~n=2,\\
C\gamma\varepsilon^{1-\frac{1}{k}}|\log\varepsilon|,\quad&\mbox{if}~~n=3,
\end{cases}
\end{align*}
and
$$|\mathrm{II}_{3}|\leq\left|2\int_{\Omega^{*}_{R_0}\setminus\Omega^{*}_{\varepsilon^{\gamma}}}\nabla{v}^{*}_{1}\cdot\nabla(v_{1}-v_{1}^{*})\right|\leq\begin{cases}
C\varepsilon^{1/2(1-\frac{1}{k})-\gamma},&\mbox{if}~~n=2,\\
C\gamma\varepsilon^{1/2(1-\frac{1}{k})}|\log\varepsilon|,\quad&\mbox{if}~~n=3.
\end{cases}$$
Now, we use the explicit function $\bar{u}^{*}$ to approximate $\nabla{v}_{1}^{*}$. Using \eqref{nablau_starbar} and \eqref{nablaw_0star}, a similar argument as in $I$ yields
\begin{align*}
\mathrm{II}_4
=\int_{\Omega^{*}_{R_0}\setminus\Omega^{*}_{\varepsilon^{\gamma}}}|\nabla{v}^{*}_{1}|^{2}=&\int_{\Omega^{*}_{R_{0}}\setminus\Omega^{*}_{\varepsilon^{\gamma}}}|\nabla\bar{u}^{*}|^{2}
+2\int_{\Omega^{*}_{R_{0}}\setminus\Omega^{*}_{\varepsilon^{\gamma}}}\nabla\bar{u}^{*}\cdot\nabla(v^{*}_{1}-\bar{u}^{*})
+\int_{\Omega^{*}_{R_{0}}\setminus\Omega^{*}_{\varepsilon^{\gamma}}}|\nabla(v^{*}_{1}-\bar{u}^{*})|^{2}\\
=&\int_{\Omega^{*}_{R_{0}}\setminus\Omega^{*}_{\varepsilon^{\gamma}}}|\partial_{x_{n}}\bar{u}^{*}|^{2}
+ A_{1}+O(\varepsilon^{(n-1)\gamma})\\
=&\int_{R_0>|x'|>\varepsilon^{\gamma}}\frac{dx'}{h_{1}(x')-h_{2}(x')}+ A_{1}+O(\varepsilon^{(n-1)\gamma}),
\end{align*}
where
\begin{align*}
A_{1}:=&2\int_{\Omega^{*}_{R_{0}}}\nabla\bar{u}^{*}\cdot\nabla(v^{*}_{1}-\bar{u}^{*})
+\int_{\Omega^{*}_{R_{0}}}\left(|\nabla(v^{*}_{1}-\bar{u}^{*})|^{2}+|\partial_{x'}\bar{u}^{*}|^{2}\right)
\end{align*}
is independent of $\varepsilon$.

For $0<\gamma\leq1/4$, $\varepsilon^{1+(n-5)\gamma} \le \varepsilon^{(n-1)\gamma}$,
it follows from these estimates above that
\begin{align*}
\mathrm{II}= \int_{R_0>|x'|>\varepsilon^{\gamma}}\frac{dx'}{h_{1}(x')-h_{2}(x')}+ A_{1}+O(\varepsilon^{(n-1)\gamma})+
\begin{cases}
O\left( \varepsilon^{1/2(1-\frac{1}{k})-\gamma} \right),&\mbox{if}~~n=2,\\
O\left( \varepsilon^{1/2(1-\frac{1}{k})}|\log\varepsilon| \right),\quad&\mbox{if}~~n=3.
\end{cases}
\end{align*}

(iii) For term $\mathrm{III}$, since
$$\Delta(v_{1}-v_{1}^{*})=0,\quad\mbox{in}~~\Omega\setminus{\big(D_{1}\cup{D}_{1}^{*}\cup{D}_{2}\cup{D}_{2}^{*}\cup\Omega_{R_{0}}\big)},$$
and
$$0<v_{1},v_{1}^{*}<1,\quad\mbox{in}~~\Omega\setminus{\big(D_{1}\cup{D}_{1}^{*}\cup{D}_{2}\cup{D}_{2}^{*}\cup\Omega_{R_{0}}\big)},$$
it follows that provided $\partial{D}_{1}^{*}$, $\partial{D}_{2}^{*}$ and $\partial \Omega$ are of $C^{k,1}$, $k \geq 3,$
$$|\nabla^{k}(v_{1}-v_{1}^{*})|\leq\,C(k),\quad\mbox{in}~~\Omega\setminus{\big(D_{1}\cup{D}_{1}^{*}\cup{D}_{2}\cup{D}_{2}^{*}\cup\Omega_{R_{0}}\big)},$$
where $C(k)$ is independent of $\varepsilon$. By an interpolation inequality with \eqref{v1-v1*}, we have
\begin{equation}\label{v1-v1*outside}
|\nabla(v_{1}-v_{1}^{*})|\leq\,C
\varepsilon^{1/2(1-\frac{1}{k})},\quad\mbox{in}~~\Omega\setminus{\big(D_{1}\cup{D}_{1}^{*}\cup{D}_{2}\cup{D}_{2}^{*}\cup\Omega_{R_{0}}\big)}.
\end{equation}
In view of the boundedness of $|\nabla{v}_{1}|$ in $(D_{1}^{*}\cup{D}_{2}^{*})\setminus(D_{1}\cup{D}_{2}\cup\Omega_{R_{0}})$ and  $(D_{1}\cup{D}_{2})\setminus(D_{1}^{*}\cup{D}_{2}^{*})$, and the fact that the volume of $(D_{1}^{*}\cup{D}_{2}^{*})\setminus(D_{1}\cup{D}_{2}\cup\Omega_{R_{0}})$ and $(D_{1}\cup{D}_{2})\setminus(D_{1}^{*}\cup{D}_{2}^{*})$ is less than $C\varepsilon$, and by using \eqref{v1-v1*outside}, we have
\begin{align*}%\label{nablav1_part1}
\mathrm{III}=&\int_{\Omega\setminus{\big(D_{1}\cup{D}_{1}^{*}\cup{D}_{2}\cup{D}_{2}^{*}\cup\Omega_{R_{0}}\big)}}|\nabla{v}_{1}|^{2}+O(\varepsilon)\nonumber\\
=&\int_{\Omega\setminus{\big(D_{1}\cup{D}_{1}^{*}\cup{D}_{2}\cup{D}_{2}^{*}\cup\Omega_{R_{0}}\big)}}|\nabla{v}_{1}^{*}|^{2}
+2\int_{\Omega\setminus{\big(D_{1}\cup{D}_{1}^{*}\cup{D}_{2}\cup{D}_{2}^{*}\cup\Omega_{R_{0}}\big)}}\nabla{v}_{1}^{*}\nabla(v_{1}-v_{1}^{*})\nonumber\\
&+\int_{\Omega\setminus{\big(D_{1}\cup{D}_{1}^{*}\cup{D}_{2}\cup{D}_{2}^{*}\cup\Omega_{R_{0}}\big)}}|\nabla(v_{1}-v_{1}^{*})|^{2}+O(\varepsilon)\nonumber\\
=&\int_{\widetilde{\Omega}^{*}\setminus\Omega^{*}_{R_0}}|\nabla{v}_{1}^{*}|^{2}+
O\left(\varepsilon^{1/2(1-\frac{1}{k})}\right).
\end{align*}

Now combining (i) (ii) and (iii) and using $0<\gamma\leq1/4$, we obtain
\begin{align}\label{energy_v1_refined}
\begin{split}
\int_{\widetilde{\Omega}}|\nabla{v}_{1}|^{2} =& \int_{R_0>|x'|>\varepsilon^{\gamma}}\frac{dx'}{h_{1}(x')-h_{2}(x')} + \int_{|x'|<\varepsilon^{\gamma}}\frac{dx'}{\varepsilon+h_{1}(x')-h_{2}(x')} \\
&+A_{2}+ O\left(\varepsilon^{(n-1)\gamma}\right)+\begin{cases}
O\left( \varepsilon^{1/2(1-\frac{1}{k})-\gamma} \right), &\mbox{if}~n=2;\\
O\left( \varepsilon^{1/2(1-\frac{1}{k})}|\log\varepsilon| \right), &\mbox{if}~n=3,
\end{cases}
\end{split}
\end{align}
where
$$A_{2}:=\int_{\widetilde{\Omega}^{*}\setminus\Omega^{*}_{R_{0}}}|\nabla{v}_{1}^{*}|^{2}+A_{1}.$$

\noindent{\bf STEP 2.} 
After a rotation of the coordinates if necessary, we assume that
\begin{equation}\label{h1h23}
h_{1}(x')-h_{2}(x')=\sum_{j=1}^{n-1}\frac{\lambda_{j}}{2}x_{j}^{2}+ \sum_{|\alpha| = 3}C_\alpha x'^\alpha + O(|x'|^4),\quad\,|x'|\leq\,R_{0},
\end{equation}
where $\mathrm{diag}(\lambda_{1},\cdots,\lambda_{n-1})=\nabla_{x'}^{2}(h_{1}-h_{2})(0')$, $C_\alpha$ are some constants, $\alpha$ is an $(n-1)$-dimensional multi-index. We call $\lambda_{1},\cdots,\lambda_{n-1}$ the relative principal curvatures of $\partial{D}_{1}$ and $\partial{D}_{2}$. 

To evaluate the first two terms in \eqref{energy_v1_refined}, we would like to replace $h_1(x')-h_2(x')$ by the quadratic polynomial $\sum_{j = 1}^{n-1} \frac{\lambda_j}{2} x_j^2$. First, under the assumption \eqref{h1h20}--\eqref{h1h21} and \eqref{h1h23}, we have
\begin{align*}
&\int_{R_0>|x'|>\varepsilon^{\gamma}}\frac{dx'}{h_{1}(x')-h_{2}(x')} - \int_{R_0>|x'|>\varepsilon^{\gamma}}\frac{dx'}{\sum_{j = 1}^{n-1} \frac{\lambda_j}{2} x_j^2}\\
=& \int_{R_0>|x'|>\varepsilon^{\gamma}}\left[\frac{1}{\sum_{j=1}^{n-1}\frac{\lambda_{j}}{2}x_{j}^{2}+ \sum_{|\alpha| = 3}C_\alpha x'^\alpha + O(|x'|^4)} - \frac{1}{\sum_{j = 1}^{n-1} \frac{\lambda_j}{2} x_j^2}\right]dx'\\
=& \int_{R_0>|x'|>\varepsilon^{\gamma}}\frac{1}{\sum_{j = 1}^{n-1} \frac{\lambda_j}{2} x_j^2} \left[ \left( 1 + \frac{\sum_{|\alpha| = 3}C_\alpha x'^\alpha}{\sum_{j = 1}^{n-1} \frac{\lambda_j}{2} x_j^2} + O(|x'|^2) \right)^{-1} - 1\right]\, dx'\\
=& \int_{R_0>|x'|>\varepsilon^{\gamma}}\frac{1}{\sum_{j = 1}^{n-1} \frac{\lambda_j}{2} x_j^2} \left[ \left( 1 - \frac{\sum_{|\alpha| = 3}C_\alpha x'^\alpha}{\sum_{j = 1}^{n-1} \frac{\lambda_j}{2} x_j^2} + O(|x'|^2) \right) - 1\right]\, dx',
\end{align*}
where in last line we use Taylor expansion due to the smallness of $R_0$. Note that $\frac{\sum_{\alpha}C_\alpha x'^\alpha}{\sum_j  (\lambda_j /2) x_j^2}$ is odd and the integrating domain is symmetric, we have
$$\int_{R_0>|x'|>\varepsilon^{\gamma}}\frac{dx'}{h_{1}(x')-h_{2}(x')} - \int_{R_0>|x'|>\varepsilon^{\gamma}}\frac{dx'}{\sum_{j = 1}^{n-1} \frac{\lambda_j}{2} x_j^2} = \int_{R_0>|x'|>\varepsilon^{\gamma}} O(1) \, dx' = \tilde{C} + O\left(\varepsilon^{(n-1)\gamma}\right),$$
where $\tilde{C}$ is some constant depending on $n, R_0, \lambda_j$ but not $\varepsilon$. Similarly, we have
$$\int_{|x'|<\varepsilon^{\gamma}}\frac{dx'}{\varepsilon+h_{1}(x')-h_{2}(x')} - \int_{|x'|<\varepsilon^{\gamma}}\frac{dx'}{\varepsilon + \sum_{j = 1}^{n-1} \frac{\lambda_j}{2} x_j^2} =\int_{|x'|<\varepsilon^{\gamma}} O(1) \, dx' = O\left(\varepsilon^{(n-1)\gamma}\right).$$
Therefore, \eqref{energy_v1_refined} becomes
\begin{align}\label{energy_v1_refined2}
\begin{split}
\int_{\widetilde{\Omega}}|\nabla{v}_{1}|^{2} =& \int_{R_0>|x'|>\varepsilon^{\gamma}}\frac{dx'}{\sum_{j = 1}^{n-1} \frac{\lambda_j}{2} x_j^2} + \int_{|x'|<\varepsilon^{\gamma}}\frac{dx'}{\varepsilon+\sum_{j = 1}^{n-1} \frac{\lambda_j}{2} x_j^2} \\
&+A_{3}+ O\left(\varepsilon^{(n-1)\gamma}\right)+\begin{cases}
O\left( \varepsilon^{1/2(1-\frac{1}{k})-\gamma} \right), &\mbox{if}~n=2;\\
O\left( \varepsilon^{1/2(1-\frac{1}{k})}|\log\varepsilon| \right), &\mbox{if}~n=3,
\end{cases}
\end{split}
\end{align}
where
$$A_{3}: = A_{2} + \tilde{C}.$$

\noindent{\bf STEP 3.} Now we deal with the first two explicit terms in \eqref{energy_v1_refined2}. 

(i) For $n=2$, we have
\begin{align*}
&2 \left( \int_{\varepsilon^\gamma}^{R_0} \frac{dx_1}{\frac{\lambda_1}{2}x_1^2 } + \int_0^{\varepsilon^\gamma} \frac{dx_1}{\varepsilon + \frac{\lambda_1}{2}x_1^2 } \right) \\
=& 2 \left( \int_{\varepsilon^\gamma}^{R_0} \frac{dx_1}{\frac{\lambda_1}{2}x_1^2 } -  \int_{\varepsilon^\gamma}^{\infty} \frac{dx_1}{\frac{\lambda_1}{2}x_1^2 } \right) + 2 \int_0^\infty  \frac{dx_1}{\varepsilon +\frac{\lambda_1}{2}x_1^2 }+O\left( \varepsilon^{1/2(1-\frac{1}{k})-\gamma} \right)\\
=& -\frac{4}{\lambda_1 R_0} +\frac{1}{\rho_n(\varepsilon)} \frac{\sqrt{2}\pi}{\sqrt{\lambda_1}}  +O\left( \varepsilon^{1/2(1-\frac{1}{k})-\gamma} \right),
\end{align*}
where we use in second line,
$$\left| \int_{\varepsilon^\gamma}^\infty \frac{1}{\varepsilon +\frac{\lambda_1}{2}x_1^2 } - \frac{1} {\frac{\lambda_1}{2}x_1^2} \, dx_1 \right| \le C \varepsilon \int_{\varepsilon^\gamma}^\infty \frac{dx_1}{x_1^4} = O\left(\varepsilon^{1 - 3\gamma}\right) \le O\left( \varepsilon^{1/2(1-\frac{1}{k})-\gamma} \right).$$
Therefore,
$$\int_{\widetilde{\Omega}}|\nabla{v}_{1}|^{2} = \frac{1}{\rho_n(\varepsilon)} \frac{\sqrt{2}\pi}{\sqrt{\lambda_1}}+(A_{3}-\frac{4}{\lambda_1 R_0})+O\left(\varepsilon^{(n-1)\gamma}\right)+O\left( \varepsilon^{1/2(1-\frac{1}{k})-\gamma} \right).$$

(ii) For $n=3$,
\begin{align}\label{n=3part1}
\begin{split}
&  \int_{\varepsilon^\gamma < |x'| < R_0} \frac{dx'}{\frac{\lambda_1}{2}x_1^2 + \frac{\lambda_2}{2}x_2^2} +  \int_{|x'| < \varepsilon^\gamma} \frac{dx'}{\varepsilon + \frac{\lambda_1}{2}x_1^2 + \frac{\lambda_2}{2}x_2^2} \\
=& \int_{|x'| <R_0} \frac{dx'}{\varepsilon + \frac{\lambda_1}{2}x_1^2 + \frac{\lambda_2}{2}x_2^2}+ O\left( \varepsilon^{1/2(1-\frac{1}{k})}|\log\varepsilon| \right),
\end{split}
\end{align}
where we used that
\begin{align*}
\left|\int_{\varepsilon^\gamma < |x'| < R_0} \frac{1}{\varepsilon+ \frac{\lambda_1}{2}x_1^2 + \frac{\lambda_2}{2}x_2^2} -  \frac{1}{\frac{\lambda_1}{2}x_1^2 + \frac{\lambda_2}{2}x_2^2} \,dx' \right| \le& C\varepsilon \int_{\varepsilon^\gamma < |x'| < R_0} \frac{dx'}{|x'|^4} = O\left(\varepsilon^{1-2\gamma} \right)\\
\le& O\left( \varepsilon^{1/2(1-\frac{1}{k})}|\log\varepsilon| \right).
\end{align*}
Denote $R(\theta):= R_0 (\frac{2}{\lambda_1} \cos^2 \theta + \frac{2}{\lambda_2} \sin^2 \theta)^{-1/2}$. After a change of variables, the first term of \eqref{n=3part1} becomes
\begin{align}\label{n=3part2}
\begin{split}
\int_{|x'| <R_0} \frac{dx'}{\varepsilon + \frac{\lambda_1}{2}x_1^2 + \frac{\lambda_2}{2}x_2^2} =& \frac{2}{\sqrt{\lambda_1 \lambda_2}} \int_0^{2\pi} \int_0^{R(\theta)} \frac{r }{\varepsilon + r^2} \, drd\theta\\
=& \left. \frac{1}{\sqrt{\lambda_1 \lambda_2}} \int_0^{2\pi} \ln (\varepsilon + r^2) \right|_{r = 0}^{R(\theta)} \, d\theta\\
=&\frac{1}{\rho_n(\varepsilon)} \frac{2\pi}{\sqrt{\lambda_1 \lambda_2}} +  \frac{1}{\sqrt{\lambda_1 \lambda_2}} \int_0^{2\pi} \ln (R(\theta)^2) + \ln(1 + \frac{\varepsilon}{R(\theta)^2}) \, d\theta\\
=&\frac{1}{\rho_n(\varepsilon)} \frac{2\pi}{\sqrt{\lambda_1 \lambda_2}} +  \frac{2}{\sqrt{\lambda_1 \lambda_2}} \int_0^{2\pi} \ln R(\theta)\, d\theta + O(\varepsilon),
\end{split}
\end{align} 
where we use the fact that $R(\theta)^2$ has a positive lower bound that is greater than $\varepsilon$, and the Taylor expansion of $\ln (1+x)$, for $|x| < 1$. Combining \eqref{n=3part1} and \eqref{n=3part2}, we conclude that for $n=3$,
\begin{align*}
\int_{\widetilde{\Omega}}|\nabla{v}_{1}|^{2} =&\frac{1}{\rho_n(\varepsilon)} \frac{2\pi}{\sqrt{\lambda_1 \lambda_2}} + \left(A_3 +\frac{2}{\sqrt{\lambda_1 \lambda_2}} \int_0^{2\pi} \ln R(\theta)\, d\theta \right)\\
&+ O\left(\varepsilon^{(n-1)\gamma}\right) + O\left( \varepsilon^{1/2(1-\frac{1}{k})}|\log\varepsilon| \right).\\
\end{align*}

We now define
$$\kappa_n:=
\begin{cases}
\frac{\sqrt{2}\pi}{\sqrt{\lambda_{1}}},&n=2,\\
\frac{2\pi}{\sqrt{\lambda_{1}\lambda_{2}}},&n=3,
\end{cases}
\quad \text{ and }
M_{1}:=
\begin{cases}
A_{3} - \frac{4}{\lambda_1 R_0},&n=2,\\
A_{3} + \frac{2}{\sqrt{\lambda_1 \lambda_2}} \int_0^{2\pi} \ln R(\theta)\,d\theta,&n=3.
\end{cases}
$$
Taking $\gamma=1/4$, $\varepsilon^{(n-1)\gamma},  \varepsilon^{1/2(1-\frac{1}{k})-\gamma}$ (or $\varepsilon^{1/2(1-\frac{1}{k})}|\log\varepsilon|$) are
smaller than $\varepsilon^{\frac{1}{4}-\frac{1}{2k}}$ (or $\varepsilon^{\frac{1}{2}-\frac{1}{2k}}|\log \varepsilon|$), and \eqref{a1i2} is proved. It is not difficult to prove that $M_{1}$ is independent of $R_{0}$. If not, suppose that there exist $M_{1}(R_{0})$ and $M_{1}(\tilde{R}_{0})$, both independent of $\varepsilon$, such that \eqref{a1i2} holds, then
$$M_{1}(R_{0})-M_{1}(\tilde{R}_{0})=\left\{ \begin{aligned}
&O \left(\varepsilon^{\frac{1}{4}-\frac{1}{2k}} \right), && \mbox{if}~n = 2,\\
&O\left(\varepsilon^{\frac{1}{2}-\frac{1}{2k}}|\log \varepsilon|\right), &&\mbox{if}~ n = 3,
\end{aligned}\right.$$
which implies that $M_{1}(R_{0})=M_{1}(\tilde{R}_{0})$.
\end{proof}

%\vspace{.5cm}

%%%%%%%%%%%%%%%%%%%%%%%%%%%%%%%%%%%%%%%%%%%%%%%%%%%%%%%%%%%%
%%%%%%%%%%%%%%%%%%%%%%%%%%%%%%%%%%%%%%%%%%%%%%%%%%%%%%%%%%%%%%

%


\begin{thebibliography}{99}



\bibitem{akl} H. Ammari; H. Kang; M. Lim, Gradient estimates to the conductivity problem. Math.
Ann. 332 (2005), 277-286.

\bibitem{ackly} H. Ammari; G. Ciraolo; H. Kang; H. Lee; K. Yun, Spectral analysis of the Neumann-Poincar\'e operator and characterization of the stress concentration in anti-plane elasticity. Arch. Ration. Mech. Anal.  208  (2013),  275-304.


\bibitem{adkl}  H. Ammari; H. Dassios; H. Kang; M. Lim, Estimates for the electric field in the
presence of adjacent perfectly conducting spheres. Quat. Appl. Math. 65
(2007), 339-355.

\bibitem{aklll}  H. Ammari; H. Kang; H. Lee; J. Lee; M. Lim, Optimal estimates for the electrical
field in two dimensions. J. Math. Pures Appl. 88 (2007), 307-324.


\bibitem{akllz} H. Ammari; H. Kang; H. Lee; M. Lim; H. Zribi, Decomposition theorems and fine estimates for electrical fields in the presence of closely located circular inclusions. J.
Differential Equations 247 (2009), 2897-2912.

\bibitem{basl} I. Babu\u{s}ka; B. Andersson; P. Smith; K.
Levin, Damage analysis of fiber composites. I. Statistical analysis
on fiber scale. Comput. Methods Appl. Mech. Engrg. 172 (1999), 27-77.


\bibitem{bly1} E. Bao; Y.Y. Li; B. Yin, Gradient estimates for the perfect conductivity problem. Arch. Ration. Mech. Anal. 193 (2009), 195-226.

\bibitem{bly2} E. Bao; Y.Y. Li; B. Yin, Gradient estimates for the perfect and insulated conductivity problems with multiple inclusions. Comm. Partial Differential Equations 35 (2010), 1982-2006.

\bibitem{bjl} J.G. Bao; J.H. Ju; H.G. Li, Optimal boundary gradient estimates for Lam\'e systems with partially infinite coefficients. Adv. Math. 314 (2017), 583-629. 

\bibitem{bll} J.G. Bao; H.G. Li; Y.Y. Li, Gradient estimates for solutions of the Lam\'e system with partially infinite coefficients, Arch. Ration. Mech. Anal.  215 (2015), no. 1, 307-351.

\bibitem{bll2} J.G. Bao; H.G. Li; Y.Y. Li, Gradient estimates for solutions of the Lam\'e system with partially infinite coefficients in dimensions greater than two. Adv. Math. 305 (2017), 298-338.

\bibitem{bgn} L. Berlyand; Y. Gorb; A. Novikov, Fictitious fluid approach and
anomalous blow-up of the dissipation rate in a 2D model of concentrated suspensions, Arch. Rat. Mech. Anal., 193 (2009), no. 3, 585-622.

\bibitem{bt} E. Bonnetier; F. Triki, On the spectrum of the Poincar\'e variational problem for two close-to-touching inclusions in 2D. Arch. Ration. Mech. Anal. 209 (2013), no. 2, 541-567.

\bibitem{bv} E. Bonnetier; M. Vogelius, An elliptic regularity result for a composite medium with ``touching'' fibers of circular cross-section. SIAM J. Math. Anal. 31 (2000), 651-677.

\bibitem{bc}  B. Budiansky; G.F. Carrier, High shear stresses in stiff fiber composites. J. App. Mech. 51 (1984), 733-735.

\bibitem{dl} H.J. Dong; H.G. Li, Optimal estimates for the conductivity problem by Green's function method.  Arch. Ration. Mech. Anal. 231 (2019), no. 3, 1427-1453.

\bibitem{gorb2} Y. Gorb, Singular behavior of electric field of high-contrast concentrated composites. Multiscale Model. Simul. 13 (2015), no. 4, 1312-1326.

\bibitem{gn} Y. Gorb; A. Novikov, Blow-up of solutions to a $p$-Laplace equation. Multiscale Model. Simul. 10 (2012), no. 3, 727-743.

\bibitem{gt} D. Gilbarg, N.S. Trudinger: Elliptic partial differential equations of second order. Reprint of the 1998 edition. Classics in Mathematics. Springer-Verlag, Berlin, 2001. xiv+517 pp. ISBN: 3-540-41160-7.

\bibitem{kly}  H. Kang; M. Lim; K. Yun, Asymptotics and computation of the solution to the conductivity equation in the presence of adjacent inclusions with extreme conductivities. J. Math. Pures Appl. (9) 99 (2013), 234-249.

\bibitem{kly2}  H. Kang; M. Lim; K. Yun,  Characterization of the electric field concentration between two adjacent spherical perfect conductors.  SIAM J. Appl. Math. 74 (2014), 125-146.

\bibitem{keller} J.B. Keller, Stresses in narrow regions, Trans. ASME J. Appl. Mech. 60 (1993), 1054-1056.

\bibitem{llby} H.G. Li; Y.Y. Li; E.S. Bao; B. Yin, Derivative estimates of solutions of elliptic systems in narrow regions. Quart. Appl. Math. 72 (2014), 589-596.

\bibitem{ll} H.G. Li; Y.Y. Li, Gradient estimates for parabolic systems from composite material. Sci. China Math. 60 (2017), no. 11, 2011-2052.

\bibitem{lx} H.G. Li; L.J.Xu, Optimal estimates for the perfect conductivity problem with inclusions close to the boundary. SIAM J. Math. Anal. 49 (2017), no. 4, 3125-3142.

\bibitem{ln}  Y.Y. Li; L. Nirenberg, Estimates for elliptic system from composite material. Comm. Pure Appl. Math. 56 (2003), 892-925.


\bibitem{lv} Y.Y. Li; M. Vogelius, Gradient stimates for solutions to divergence form elliptic equations with discontinuous coefficients. Arch. Rational Mech. Anal. 153 (2000), 91-151.

\bibitem{ly}M. Lim; K. Yun, Lim, Blow-up of electric fields between closely spaced spherical perfect conductors. Comm. Partial Differential Equations 34 (2009), no. 10-12, 1287-1315.

\bibitem{lyu} M. Lim; S. Yu, Asymptotics of the solution to the conductivity equation in the presence of adjacent circular inclusions with finite conductivities. J. Math. Anal. Appl.  421  (2015),  no. 1, 131-156.

\bibitem{m} X. Markenscoff, Stress amplification in vanishingly small geometries. Computational Mechanics 19  (1996), 77-83.

%\bibitem{wang} Y.P. Wang, An unpublished note.


\bibitem{y1} K. Yun, Estimates for electric fields blown up between closely adjacent conductors with arbitrary shape. SIAM J. Appl. Math. 67 (2007),  714-730.

\bibitem{y2} K. Yun, Optimal bound on high stresses occurring between stiff fibers with arbitrary shaped cross-sections. J. Math. Anal. Appl. 350 (2009), 306-312.

\end{thebibliography}
\end{document}